\documentclass{amsart}
\usepackage{amsmath}
\usepackage{amsthm}
\usepackage{amsfonts}
\usepackage{mathrsfs}
\usepackage[margin=1.4in]{geometry}
\usepackage{verbatim}
\usepackage{url}
\usepackage{color}

\title[Weighted Generating Functions for {Type~II} Lattices and Codes]
 {Weighted Generating Functions \\ for Type~II Lattices and Codes}
\thanks{The authors thank Zachary Abel, Henry Cohn, John H.\ Conway,
 John~F.\ Duncan, Benedict~H.\ Gross, Abhinav Kumar, Barry Mazur,
 Gabriele Nebe, Ken Ono, Vera Pless, Eric~M.\ Rains, and Shrenik Shah
 for helpful comments and suggestions.  During parts of this research,
 Elkies was supported by NSF grants DMS-0501029 and DMS-1100511,
 and Kominers was supported by
 the Harvard College Program for Research in Science and Engineering (PRISE),
 a Harvard Mathematics Department Highbridge Fellowship,
 an NSF Graduate Research Fellowship,
 a Yahoo!\ Key Scientific Challenges Program Fellowship,
 and an AMS-Simons Travel Grant.}
\thanks{This work includes a part of the second author's undergraduate
 thesis~\cite{Kominers:thesis}.}
\author{Noam D.~Elkies}
\address{Department of Mathematics, Harvard University\newline\indent
 One Oxford Street\newline\indent Cambridge, MA 02138}
\email{elkies@math.harvard.edu}

\author{Scott Duke Kominers}
\address{
Becker Friedman Institute for Research in Economics\newline\indent
University of Chicago\newline\indent
1126 East 59th Street\newline\indent
Chicago, IL 60637}
\email{skominers@gmail.com}


\newtheorem{theorem}{Theorem}[section]
\newtheorem{lemma}[theorem]{Lemma}

\newtheorem{cor}[theorem]{Corollary}
\newtheorem{prop}[theorem]{Proposition}

\theoremstyle{definition}

\newcommand\C{\mathbb{C}}
\newcommand\F{\mathbb{F}}
\newcommand\R{\mathbb{R}}
\newcommand\Z{\mathbb{Z}}

\newcommand\somedesign{D}

\newcommand\dhpk{R}
\newcommand\up[1]{\upp{#1}'}
\newcommand\upp[1]{u_{#1}}
\newcommand\dhpkspace{\mathscr{O}}
\newcommand\proj[2]{\pi_{#1}(#2)}

\DeclareMathOperator\Sym{Sym}

\newcommand\setsep{:}
\newcommand\thmnamesep{--}

\newcommand\defeq{:=}
\newcommand\thmmulticitesep{;\ }
\newcommand\dispoplus{\bigoplus}

\newcommand\defn[1]{\emph{#1\/}}
\newcommand\earlyterm[1]{\emph{#1\/}}

\newcommand\slmat[4]{\left(\begin{smallmatrix}#1&#2\\#3&#4\end{smallmatrix}\right)}
\newcommand\sslmat[4]{\left(\begin{smallmatrix}#1&#2\\#3&#4\end{smallmatrix}\right)}

\newcommand\GL{\mathrm{GL}}

\newcommand\Slalg{\mathfrak{sl}}
\newcommand\idsl{\sslmat{1}{0}{0}{1}}

\newcommand\commut[2]{\left[#1,#2\right]}

\newcommand\inlinecup{\bigcup}

\newcommand\zonalify[2]{\mathscr{Z}\! #1}

\newcommand\polyspace{\mathscr{P}}
\newcommand\polyspaced[1]{\polyspace_{#1}}

\newcommand\dpolyspace{\mathscr{D}}
\newcommand\dpolyspaced[1]{\dpolyspace_{#1}}
\newcommand\dzpolyspace{\zonalify{\dpolyspace}{\cfixed}}
\newcommand\dzpolyspaced[1]{\dzpolyspace_{#1}}
\newcommand\dfpolyspaced[2]{\dpolyspaced{#1}^{#2}}
\newcommand\dharmspace{\dpolyspace^0}
\newcommand\dharmspaced[1]{\dharmspace_{#1}}
\newcommand\dzharmspace{\zonalify{\dharmspace}{\cfixed}}
\newcommand\dzharmspaced[1]{\dzharmspace_{#1}}

\newcommand\oper[1]{\mathsf{#1}}
\newcommand\goper[1]{\mathsf{#1}}

\newcommand\Goper[1]{\goper{G}_{#1}}
\newcommand\Xoper{\oper{X}}
\newcommand\Hoper{\oper{H}}
\newcommand\Yoper{\oper{Y}}
\newcommand\Woper{\goper{W}}
\newcommand\Wopero{\goper{V}}
\newcommand\FToper{\goper{T}}

\newcommand\pset[1]{#1'}
\newcommand\Xoperp{\pset{\Xoper}}
\newcommand\Hoperp{\pset{\Hoper}}
\newcommand\Yoperp{\pset{\Yoper}}

\newcommand\nset[1]{\widetilde{#1}}
\newcommand\FTopern{\nset{\FToper}}
\newcommand\Hopern{\nset{\Hoper}}
\newcommand\Wopern{\nset{\Woper}}
\newcommand\Woperon{\nset{\Wopero}}
\newcommand\Xoperpn{\nset{\Xoperp}}
\newcommand\Hoperpn{\nset{\Hoperp}}
\newcommand\Yoperpn{\nset{\Yoperp}}
\newcommand\OPopern{\bigl(\Woperon^{-1}\FTopern\Wopern\bigr)}

\newcommand\konstant{b}

\newcommand\act[2]{\left(#1\right)(#2)}

\newcommand\otimesdots{\otimes\cdots\otimes}
\newcommand\mono[2]{\left(\begin{smallmatrix} #1\\#2\end{smallmatrix}\right)}
\newcommand\monom[4]{\mono{#1}{#2}\otimesdots\mono{#3}{#4}}
\newcommand\monomone[2]{\mono{#1}{#2}^{\otimes n}}

\newcommand\slmatn[4]{\slmat{#1}{#2}{#3}{#4}^{\otimes n}}
\newcommand\slmate[4]{\idsl\otimesdots\slmat{#1}{#2}{#3}{#4}\otimesdots\idsl}

\newcommand\slmatone[1]{\idsl\otimesdots #1\otimesdots\idsl}

\newcommand\sumslmateone[1]{\sum\slmatone{#1}}

\newcommand\dzoharmd[2]{\dhp_{#1;#2}}

\newcommand\ft[1]{\hat {#1}}

\newcommand\htr[1]{\ft{#1}}

\newcommand\funddom[1]{\mathcal{D}}

\newcommand\gleasonphi{\varphi_8}
\newcommand\gleasonxi{\xi_{24}}
\newcommand\gleasonpsi[1]{\psi_{#1}}
 \newcommand\gleasondelta{\delta_8}

\newcommand\charf{\chi^{\phantom{D}}}

\newcommand\hammingsphere{\sigma}

\newcommand\module{V}

\newcommand\duallat{*}
\newcommand\dualcode{\bot}

\newcommand\cutoff[1]{\mathrm{t}(#1)}
\newcommand\minLoz[1]{m_0}
\newcommand\minCoz[1]{w_0}

\newcommand\monomfunct{g}
\newcommand\dhp{Q}
\newcommand\dhpdegone[2]{\dhp_{1,#1,#2}}

\newcommand\fixed[1]{\dot{#1}}
\newcommand\xfixed{\fixed{x}}
\newcommand\cfixed{\fixed{v}}
\newcommand\ccfixed{\fixed{c}}
\newcommand\wtcfixed{\wt(\cfixed)}
\newcommand\invgen[2]{\dhp_{#1,#2;\cfixed}}
\newcommand\coordset{\mathrm{C}}
\newcommand\coordone{\coordset_{1;\cfixed}}
\newcommand\coordzero{\coordset_{0;\cfixed}}

\newcommand\rnk[1]{n}

\DeclareMathOperator\minwt{min}
\DeclareMathOperator\wt{wt}

\newcommand\isectcode[2]{#1\cap #2}

\newcommand\latshell[2]{{#2}_{#1}}

\newcommand\wecoeff[2]{|\codeshell{#1}{#2}|}

\newcommand\inprodcountc[2]{N_{#1}(C;#2)}

\newcommand\codeshell[2]{{#2}_{#1}}
\newcommand\gencodeshell[2]{\mathcal{C}_{#1}(#2)}

\newcommand\coxcode{\eta}

\newcommand\tetrads[1]{\codeshell{4}{#1}}
\newcommand\tetradsyst[1]{\gencodeshell{4}{#1}}

\renewcommand\Im{\mathop{\mathrm{Im}}}
\renewcommand\Re{\mathop{\mathrm{Re}}}

\newcommand\inprodl[2]{\langle#1,#2\rangle}

\newcommand\inprodc[2]{(#1,#2)}
\newcommand\inprodch[2]{(#1,#2)}

\newcommand\inlinefrac[2]{#1/#2}

\newcommand\onevec{{\mathbf 1}}

\newcommand\ra{\rightarrow}
\newcommand\0{^{\phantom0}}
\newcommand\9{_{\phantom9}}
\DeclareMathOperator\Vol{Vol}
\DeclareMathOperator\Aut{Aut}
\newcommand\upperH{{\mathcal H}}
\newcommand\SL{\mathrm{SL}} % not sure why SDK had \Sl
\newcommand\sD{{\mathsf \Delta}}
\newcommand\sE{{\mathsf E}}
\newcommand\sF{{\mathsf F}}

\newcommand\Eis{{\mathcal E}}
\newcommand\rC{{\mathrm C}}
\newcommand\PP{{\mathscr P}}
\newcommand\del{\partial}
\newcommand\Binom[2]{\genfrac{(}{)}{0pt}{}{#1}{#2}}
\newcommand\Phat{{\widehat P}}
\newcommand\nufixed{\fixed{\nu}}
\newcommand\Beta{{\mathrm B}}
\newcommand\kwad{{\mathrm Q}}
\newcommand\Vee{{\mathrm V}}

\numberwithin{equation}{section}

\begin{document}
\begin{abstract}
  We give a new structural development of harmonic polynomials
  on Hamming space, and harmonic weight enumerators of binary
  linear codes, that parallels one approach to harmonic polynomials
  on Euclidean space and weighted theta functions of Euclidean lattices.
  Namely, we use the finite-dimensional representation theory of~$\Slalg_2$
  to derive a decomposition theorem for the spaces of discrete
  homogeneous polynomials in terms of the spaces of
  discrete harmonic polynomials, and prove a generalized
  MacWilliams identity for harmonic weight enumerators.
  We then present several applications of harmonic weight enumerators,
  corresponding to some uses of weighted theta functions:
  an equivalent characterization of \hbox{$t$-designs}, the
  Assmus\thmnamesep Mattson Theorem in the case of extremal Type~II codes,
  and configuration results for extremal Type~II codes
  of lengths $8$, $24$, $32$, $48$, $56$, $72$, and~$96$.
\end{abstract}
\maketitle

\section{Introduction}\label{sec:intro}
  A well-known and fruitful analogy relates
lattices~$L$ in Euclidean space $\R^n$
with linear codes~$C$\/ in binary Hamming space $\F_2^n$.
(See for instance \cite{Ebeling:lattices}, \cite{NDE:Notices},
and \cite[3.2]{SPLAG}.)
Under this analogy the theta function
\begin{equation}
\Theta_L(q) = \sum_{v\in L} q^{\inprodl{v}{v}/2}
% \xxx{yipes, overloaded $n$ in the first few lines of text!
% Maybe change $n$ to $k$ in this last equation?  ---NDE}\\
%
%  = \sum_{n\geq 0} \left(\sum_{\inprodl{v}{v}=2n} 1\right) q^n,
           = \sum_{k\geq 0} \left(\sum_{\inprodl{v}{v}=2k} 1\right) q^k,
\label{eq:Theta}
\end{equation}
a generating function that counts vectors $v\in L$ in spheres
$\{v \setsep \inprodl{v}{v}=2k \}$ about the origin, corresponds to
the weight enumerator
\begin{equation}
W_C(x,y) = \sum_{c\in C} x^{n-\wt(c)}y^{\wt(c)}
	 = \sum_{w=0}^n \left(\sum_{\wt(c)=w} 1 \right) x^{n-w}y^w,
\label{eq:W_C}
\end{equation}
a generating function that counts words $c\in C$ in Hamming spheres
$\{c \setsep \wt(c) = w \}$ about the origin.
This paper concerns a generalization of
$\Theta_L$ and $W_C$ that can be used not only to count
lattice or code elements in each sphere,
by summing the constant function~$1$
as in~\eqref{eq:Theta} and~\eqref{eq:W_C},
but also to measure their distribution,
by summing a suitable nonconstant function~$P$.
In the lattice case, $P$\/ is a harmonic polynomial on~$\R^n$,
yielding the weighted theta function
\begin{equation}
\Theta_{L,P}(q) = \sum_{v\in L} P(v) q^{\inprodl{v}{v}/2}
% = \sum_{n\geq 0} \left(\sum_{\inprodl{v}{v}=2n} P(v) \right) q^n.
  = \sum_{k\geq 0} \left(\sum_{\inprodl{v}{v}=2k} P(v) \right) q^k.
\label{eq:weighted_theta}
\end{equation}
In the code case, $P$\/ is a discrete harmonic polynomial on~$\F_2^n$,
yielding the harmonic weight enumerator\footnote{
  While the analogy between $\Theta_{L,P}$ and $W_{C,P}$ suggests
  calling $W_{C,P}$ a ``weighted weight enumerator'', the comical
  juxtaposition of the two senses of ``weight'' dissuades us from
  using that phrase.  Since Bachoc~\cite{Bachoc:binary} had already
  introduced the term ``harmonic weight enumerator'' that avoids
  this juxtaposition, we happily follow her usage.
%    \xxx{Ask Bachoc if she likewise considered but rejected
%    ``weighted weight''}
%    Done.  She didn't, actually: she knew of "theta series with
%    [harmonic|spherical] coefficients".
  }
\begin{equation}
W_{C,P}(x,y) = \sum_{c\in C} P(c) x^{n-\wt(c)}y^{\wt(c)}
	 = \sum_{w=0}^n \left(\sum_{\wt(c)=w} P(c) \right) x^{n-w}y^w.
\label{eq:W_C,P}
\end{equation}
Weighted theta functions have been used extensively to study the
configurations of lattice vectors.
But discrete harmonic polynomials and harmonic weight enumerators
are relatively unknown and rarely used.  Moreover, the known
construction of discrete harmonic polynomials~$P$, and the known proofs
of the basic properties of these~$P$\/ and of the associated $W_{C,P}$
(see \cite{Delsarte:Hahn, Bachoc:binary}), involve manipulations of
intricate combinatorial sums that are considerably harder than,
and look nothing like, the developments of their Euclidean counterparts.

Here we give a structural development of discrete harmonic polynomials and
harmonic weight enumerators that parallels the more familiar theory of
harmonic polynomials on~$\R^n$ and weighted theta functions.
In each case we use an action of the Lie algebra $\Slalg_2$ on spaces of
functions on $\R^n$ (for lattices) or on $\F_2^n$ (for codes).
While the two cases
are not completely parallel, the remaining distinctions are
inherent in the structure of Euclidean and Hamming space; for instance,
homogeneous polynomials on~$\F_2^n$ cannot be defined by $P(cv) = c^d P(v)$,
and since Hamming space is finite all the representations of $\Slalg_2$
that figure in the discrete theory are finite-dimensional.
Once we have established the new approach to discrete harmonic polynomials
and harmonic weight enumerators, we use it to give cleaner derivations
of the Assmus\thmnamesep Mattson theorem~\cite{assmusmattson} and
the Koch condition~\cite{Koch87}
on the tetrad system of a Type~II code of length~24.\footnote{
  The second of these requires only the $W_{C,P}$ for $P$ of degree~$1$,
  which coincide with Ott's
  ``local weight enumerators''~\cite{Ott:local}.
  }
Finally we outline some further applications to the configurations of
minimal-weight words in extremal Type~II codes that parallel recent
configuration results for extremal Type~II lattices.

The rest of the paper is organized as follows.  We first outline the
$\Slalg_2$ approach to harmonic polynomials on~$\R^n$ and to the
construction and basic properties of weighted theta functions,
and the connection with design properties of Type~II lattices.
In the next section we review the MacWilliams identity for weight enumerators
and Gleason's theorem for the weight enumerator of a Type~II code.
In the following three sections we use the $\Slalg_2$ theory
to develop the theory of discrete harmonic polynomials~$P$,
prove the MacWilliams identity for harmonic weight enumerators $W_{C,P}$,
and study the important special case where $P$\/ is a
``zonal harmonic polynomial'' (discrete harmonic polynomial
invariant under a subgroup $S_w \times S_{n-w}$ of the group $S_n$ of
coordinate permutations of $\F_2^n$).  The next two sections relate
these polynomials with \hbox{$t$-designs} and recover the
Assmus\thmnamesep Mattson theorem for extremal Type~II codes
and the Koch condition for Type~II codes of length~$24$.
Finally we use these techniques to show for several values of~$n$
that any extremal Type~II code of length~$n$ is generated by its
words of minimal weight, again in analogy with known results for
extremal Type~II lattices.  In an Appendix, we give a direct proof
of Gleason's theorems for self-dual codes of Type~I and~II;
certain polynomials needed to describe harmonic weight enumerators
occur naturally in the course of this proof.

While the present paper considers codes only over $\F_2$,
discrete harmonic polynomials and harmonic weight enumerators generalize to
linear codes over arbitrary finite fields~$\F_q$
(see \cite{Bachoc:non-binary}).
Our development of these notions extends to that setting too,
using representations of $\Slalg_q$ \hbox{instead} of $\Slalg_2$.
This change introduces enough new complications that we defer the
analysis to a separate paper.

\section{Weighted Theta Functions and Configurations of Type~II Lattices}\label{sec:wtf_section}
%        \subsection{Lattice-Theoretic Preliminaries}
           \subsection{Lattice-Theoretic Preliminaries}
By a \defn{lattice} in Euclidean space $\R^n$ we mean a discrete subgroup
$L \subset \R^n$ of rank~$n$; equivalently, $L$ is the \hbox{$\Z$-span} of
the columns of an invertible $n \times n$ real matrix, say~$M$\/
(which does not depend uniquely on~$L$: two such matrices $M,M'$ yield
the same $L$ iff $M^{-1}M'$ has integer entries and determinant $\pm1$).
The \defn{covolume} $\Vol(\R^n/L)$ of the lattice is then
$\left|\det M\right|$.  The \defn{dual lattice} is defined by
\begin{equation}
L^* = \{ v^* \in \R^n \setsep \forall v \in L, \inprodl{v}{v^*} \in \Z \}.
\label{eq:L*}
\end{equation}
If $L$ is the \hbox{$\Z$-span} of the columns of the invertible matrix $M$
then $L^*$ is the \hbox{$\Z$-span} of the columns of the transpose
of~$M^{-1}$; in particular $\Vol(\R^n/L^*) = \Vol(\R^n/L)^{-1}$.

If $L = L^*$ then $L$ is \defn{self-dual}.
Then $\inprodl{v}{v'} \in \Z$ for all $v,v' \in L$, and the norm map
$L \ra \Z$, $v \mapsto \inprodl{v}{v}$ reduces modulo~$2$ to a
group homomorphism $L \ra \Z / 2\Z$.  The lattice is said to be
\defn{even}\/ or \defn{of Type~II}\/ if this homomorphism is trivial,
that is, if $\inprodl{v}{v} \in 2\Z$ for all $v \in L$; otherwise
$L$ is said to be \defn{odd}\/ or \defn{of Type~I}.

\subsubsection*{Examples}
For each $n \geq 1$ the lattice $\Z^n \subset \R^n$ is of Type~I.
It is the unique Type~I lattice in~$\R^n$ for $n=1$,
and unique up to isomorphism for $n \leq 8$,
but not unique for any $n \geq 9$; there are finitely many isomorphism
classes of Type~I lattices in~$\R^n$, but the number of classes
grows rapidly with~$n$ (see for instance \cite[p.~403]{SPLAG}).

If $\R^n$ contains a Type~II lattice then $n \equiv 0 \bmod 8$
(see \cite[Chapter V]{Serre:course}).  Such a lattice may be
constructed as follows.  For any $n$ let $D_n$ be the sublattice
of~$\Z^n$ consisting of all $(x_1,\ldots,x_n)$ such that
$\sum_{j=1}^n x_j \equiv 0 \bmod 2$, and let $D_n^+$ be the union of
$D_n$ and the translate of $D_n$ by the half-integer vector
$(1/2, 1/2, \ldots, 1/2)$.  Then $D_n^+$ is:
\begin{itemize}
\item a lattice if and only if $2 \mid n$,
\item self-dual if and only if $4 \mid n$, and
\item of Type~II if and only if $8 \mid n$.
\end{itemize}
For $n=8$, this lattice $D_8^+$ coincides with the Gosset root lattice $E_8$,
which is known to be the unique Type~II lattice in~$\R^8$
up to isomorphism; we give one proof of its uniquenss at the end
of this section.\footnote{
  Serre \cite[Chapter VII]{Serre:course} uses the notation $E_n$
  for our $D_n^+$ for all $n \equiv 0 \bmod 8$, but this notation
  has not been widely adopted.  For $n\equiv 4 \bmod 8$ the Type~I lattice
  $D_n^+$ is isomorphic with $\Z^n$ if and only if $n = 4$.
  }
There are two Type~II lattices for $n=16$
(namely $E_8 \oplus E_8$ and $D_{16}^+$), and $24$ for $n=24$
(the Niemeier lattices~\cite{Niemeier:24});
for large $n \equiv 0 \bmod 8$ the number is again always finite
but grows rapidly as $n \ra \infty$
(see for instance \cite[p.~50]{SPLAG}).

\subsection{Poisson Summation}
The \defn{Poisson summation formula} is a remarkable identity relating
the sum of a function~$f$ over a lattice and the sum of the
Fourier transform of~$f$\/ over the dual lattice.
We review this formula in the case of Schwartz functions,
which is all that we need.
Recall that a \defn{Schwartz function} is a $\rC^\infty$ function
$f: \R^n \ra \C$ such that $f$\/ and all its partial derivatives decay as
$o(\inprodl{x}{x}^k)$ for all~$k$\/ as $\inprodl{x}{x} \ra \infty$.
We define the {\em Fourier transform} $\hat f: \R^n \ra \C$ by
\begin{equation}
\hat f(y) = \int_{x \in \R^n} f(x) \, e^{2\pi i \inprodl{x}{y}} \, d\mu(x);
\label{eq:fourier}
\end{equation}
$\hat f$\/ is a Schwartz function if~$f$\/ is.

\begin{theorem}[Poisson Summation Formula]\label{thm:psum}
Let $L$ be any lattice in~$\R^n$.  Then 
\begin{equation}
\sum_{x\in L} f(x) = \frac1{\Vol(\R^n/L)} \sum_{y \in L^*\9} \hat f(y) 
\label{eq:poisson}
\end{equation}
for all Schwartz functions $f: \R^n \ra \C$. 
\end{theorem}

\begin{proof} Define $F: \R^n \ra \C$\/ by
\begin{equation*}
F(z) = \sum_{x \in L} f(x+z).
\label{eq:F}
\end{equation*}
Because $f$ is Schwartz, the sum converges absolutely
to a $\rC^\infty$ function, whose value at $z=0$ is
the left-hand side of~\eqref{eq:poisson}.
Since $F(z) = F(x+z)$ for all $z \in \R^n$ and $x \in L$,
the function descends to a $\rC^\infty$ function on $\R^n/L$,
and thus has a Fourier expansion
\begin{equation}
F(z) = \sum_{y \in L^*\9} \widehat F(-y) \, e^{2\pi i \inprodl{y}{z}},
\label{eq:F_fourier}
\end{equation}
where
\begin{equation*}
\widehat F(y) = \frac1{\Vol(\R^n/L)}
\int_{z \in \R^n\9 / L} F(z) \, e^{2 \pi i \inprodl{z}{y}} \, d\mu(z).
\label{eq:Fhat}
\end{equation*}
Note that the vectors $y \in L^*$ are exactly those for which
$e^{2 \pi i \inprodl{x}{y}}$ is well-defined on $\R^n/L$.
Now choose a fundamental domain~$R$\/ for $\R^n/L$;
for instance, let $v_1,\ldots,v_n$ be generators of~$L$ and set
$
R = \{ a_1 v_1 + \cdots + a_n v_n \setsep 0 \leq a_i < 1 \}.
$
Then we have
\begin{align*}
\Vol(\R^n/L) \widehat F(y)
& =
\int_{z \in R} F(z) \, e^{2 \pi i \inprodl{y}{z}} \, d\mu(z)
\nonumber \\
& =
\int_{z \in R} \sum_{x\in L} f(x+z) \, e^{2 \pi i \inprodl{y}{z}} \, d\mu(z)
\nonumber \\
& =
\sum_{x\in L} \int_{z \in R+x} f(z) \, e^{2 \pi i \inprodl{y}{z}} \, d\mu(z)
\nonumber \\
& =
\int_{z \in \R^n} f(z) \, e^{2 \pi i \inprodl{y}{z}} \, d\mu(z)
\, = \, \hat f(y),
\label{eq:Fhat=fhat}
\end{align*}
where we used in the last step the fact that $\R^n$ is
the disjoint union of the translates $R+x$ of $R$\/ by lattice vectors.
Thus \eqref{eq:F_fourier} becomes
\begin{equation}
F(z) = \frac1{\Vol(\R^n/L)}
  \sum_{y \in L^*\9} \hat f(-y) \, e^{2\pi i \inprodl{y}{z}}.
  \label{eq:F_fourier_f}
\end{equation}
Taking $z=0$ we obtain \eqref{eq:poisson}.
\end{proof}

\subsection{Theta Functions}
Suppose now that $q$ is a real number with $0 < q < 1$.
We may then take $f(x) = q^{\inprodl{x}{x} / 2}$ and recognize
the left-hand side of~\eqref{eq:poisson} as the sum
$\Theta_L(q)$ of~\eqref{eq:Theta}.
The Poisson summation formula then yields the following
functional equation for theta functions.

\begin{prop}\label{prop:theta_func_eq}
Let $L$ be any lattice in~$\R^n$.  Then 
\begin{equation}
\Theta\0_{L^*}(e^{-2\pi t}) =
\Vol(\R^n/L) t^{-n/2} \Theta\0_L(e^{-2\pi/t})
\label{eq:Thetadual}
\end{equation}
for all $t>0$.
\end{prop}

\begin{proof}
Let $f(x) = \exp(-\pi\inprodl{x}{x}/t)$ in~\eqref{eq:poisson}.
We claim that
\begin{equation}
\hat f(y) = t^{n/2} \exp(-\pi\inprodl{y}{y}t).
\label{eq:Gaussian-hat}
\end{equation}
Indeed, choosing any orthonormal coordinates $(x_1,\ldots,x_n)$ for $\R^n$,
we see that the integral \eqref{eq:fourier} defining $\hat f(y)$
factors as
$$
\prod_{j=1}^n
  \int_{-\infty}^\infty \, e^{-\pi x_j^2/t} e^{2\pi i x_j y_j} \, dx_j,
$$
which reduces our claim to the case $n=1$,
which is the familiar definite integral
$$
 \int_{-\infty}^\infty e^{-\pi x^2/t} \, e^{2\pi i x y} \, dx
 = t^{1/2} e^{-\pi t y^2}
$$
(see for instance \cite[Example 9.43, pp.~237--238]{Rudin} or
\cite[Lemma~50.2(i), \hbox{pp.~246--247}]{Korner}).
Using these $f$\/ and $\hat f$\/
in the Poisson summation formula~\eqref{eq:poisson}
we deduce the functional equation~\eqref{eq:Thetadual}.
\end{proof}

Now suppose $L$ is a Type~II lattice.  Then $L^* = L$, so
the functional equation relates $\Theta\0_L$ to itself,
and $\Vol(\R^n/L) = 1$.
Moreover, each of the exponents $\inprodl{v}{v}/2$
occurring in the formula~\eqref{eq:Theta} is an integer,
so $\Theta_L(q)$ is a power series in~$q$ and extends to
a function on the unit disc $|q|<1$ in~$\C$.  Thus by analytic continuation
the identity $\Theta\0_L (e^{-2\pi t}) = t^{-n/2} \Theta\0_L(e^{-2\pi/t})$
holds for all $t \in \C$ of positive real part.
But $\Theta\0_L (e^{-2\pi t})$, being a power series in~$e^{-2\pi t}$,
is also invariant under $t \mapsto t+i$.
This leads us to define the function
\begin{equation}
\theta\0_L(\tau) := \Theta\0_L(e^{2\pi i \tau})
     = \sum_{v\in L} e^{\pi \inprodl{v}{v} i \tau}
\label{eq:theta}
\end{equation}
for $\tau$ in the Poincar\'e upper half-plane
$$
\upperH := \{\tau \in \C \setsep \Im(\tau) > 0\} .
$$
Then $\theta\0_L(\tau) = \theta\0_L(\tau + 1)$, and the Poisson identity
gives $\theta\0_L(\tau) = t^{-n/2} \theta\0_L(-1/\tau)$: the expected
factor of $i^{n/2}$ disappears because $n \equiv 0 \bmod 8$
for all Type~II lattices.  It follows that
\begin{equation}
\theta\0_L(\tau)
= (c\tau+d)^{-n/2} \, \theta\0_L\left(\frac{a\tau+b}{c\tau+d}\right)
\label{eq:theta_form}
\end{equation}
for all $\slmat{a}{b}{c}{d}$ in the subgroup of $\SL_2(\R)$
generated by $\slmat{1}{1}{0}{1}$ and $\slmat{0}{-1}{1}{0}$.
This subgroup is the full modular group $\SL_2(\Z)$ of integer matrices
of determinant~$1$.  (See \cite[Chapter VII]{Serre:course} for this
and the remaining results noted in this paragraph.)
The identity~\eqref{eq:theta_form} for all such $\slmat{a}{b}{c}{d}$,
together with the fact that $\theta\0_L(\tau)$ remains bounded as
$\Im(\tau) \ra \infty$ (because then $q \ra 0$),
then shows that $\theta\0_L$ is a modular form of weight $n/2$ for
$\SL_2(\Z)$.  Since $n/2 \equiv 0 \bmod 4$, this means that
$\theta\0_L$ is a polynomial in the normalized Eisenstein series
$$
\Eis_4 = \theta\0_{E_8}(\tau) = 1 + 240 \sum_{n=1}^\infty \frac{n^3 q^n}{1-q^n}
= 1 + 240 q + 2160 q^2 + 6720 q^3 + \cdots
$$
of weight~$4$ (where again $q = e^{2 \pi i \tau}$)
and the cusp form\footnote{
  That is, a modular form vanishing at all the cusps; for $\SL_2(\Z)$
  there is only one cusp, at $\Im(\tau) \ra \infty$,
  so a modular form in $\SL_2(\Z)$ is a cusp form if and only if
  its expansion as a power series in~$q$ has constant coefficient zero.
  Note that the notation of~\cite{Serre:course} diverges from the
  usual practice that we follow: our $\Eis_4$, $\Eis_6$, and $\Delta$
  are what Serre calls $E_2$, $E_3$ and $(2\pi)^{-12} \Delta$.
  (We use ``$\Eis$\/'' rather than ``$E$\,''
  to avoid confusion with the $E_8$ lattice.)
  }
$$
\Delta(\tau) = q \prod_{n=1}^\infty (1-q^n)^{24}
=  q - 24 q^2 + 252 q^3 - 1472 q^4 \cdots
$$
of weight~$12$.  Moreover the coefficient of $\Eis_4^{n/8}$
in this polynomial equals~$1$ because that coefficient is the
constant coefficient in the \hbox{$q$-expansion}, which is the
number of lattice vectors of norm~zero.

It follows for example that if $n=8$ or $n=16$ then
$\theta\0_L = \Eis_4^{n/8}$,
while if $n=8m$ with $m=3$, $4$, or~$5$
and $L$ contains no vectors $v$ with $\inprodl{v}{v} = 2$ then
$\theta\0_L = \Eis_4^m - 240m \theta_{E_8}^{m-3} \Delta$
(so for example the $q^2$ coefficient is $720m(211-40m) > 0$
and $L$ has that many vectors $v$ with $\inprodl{v}{v} = 4$).
It is known that such $L$ are unique up to isomorphism for $n=8$ and $n=24$
(the $E_8$ and Leech lattices respectively),
but there are two choices for $n=16$,
and literally millions for $n=32$ (see \cite{King:32})
and many more for $n=40$,
all with the same number of vectors of norm $2k$ for each~$k$.

More generally, given any $n=8m$
the theta series of any Type~II lattice $L$ can be written uniquely as
$\Eis_4^m + \sum_{k=1}^{\lfloor m/3 \rfloor} a_k \Delta^k \Eis_4^{m-3k}$
for some $a_k$.  If $L$ contains no vectors $v$ with
$0 < \inprodl{v}{v} \leq 2 \lfloor m/3 \rfloor$ then the $a_k$ are
uniquely determined by induction, and thus
all such lattices have the same theta series.
Such lattices~$L$ are known as \defn{extremal lattices},
and their common theta function $\theta\0_L$ is the
\defn{extremal theta function}.  
Siegel \cite{Siegel:extremal} proved that the
$q^{\lfloor m/3 \rfloor + 1}$ coefficient of $\theta\0_L$ is positive,
{}from which Mallows, Odlyzko, and Sloane~\cite{MallowsOdlyzkoSloane}
deduced that a Type~II lattice $L \subset \R^n$
has minimal norm at most $2(\lfloor m/3 \rfloor + 1)$,
with equality if and only if $L$ is extremal.

         \subsection{The Spaces of Harmonic Polynomials}
           Let $\PP$\/ be the $\C$-vector space of polynomials on~$\R^n$, and
$\PP_d$ ($d=0,1,2,\ldots$) its subspace of homogeneous polynomials
of degree~$d$, so that $\PP = \bigoplus_{d=0}^\infty \PP_d$.
The {\em Laplacian} is the differential operator defined by\footnote{
  The use of $\sD$ for this operator and $\Delta$ for the modular form
  $\eta^{24} = q \prod_{n=1}^\infty (1-q^n)^{24}$ may be unfortunate,
  but should not cause confusion, despite the similarity between the
  two symbols, because they never appear together outside this footnote.
  The alternative notation $L$ for the Laplacian would be much worse,
  as we regularly use $L$ for a lattice.
  }
\begin{equation}
\sD = \sum_{j=1}^n \frac{\del^2}{\del x_j^2} : \
\rC^\infty(\R^n) \ra \rC^\infty(\R^n), \quad
\PP \ra \PP, \quad
\PP_d \ra \PP_{d-2}.
\label{eq:sDelta}
\end{equation}
Here $x_1,\ldots,x_n$ are any orthonormal coordinates on~$\R^n$,
and $\PP_d$ is taken to be $\{0\}$ for $d<0$.
The space of \defn{harmonic polynomials} of degree~$d$\/ is then
\begin{equation}
\PP_d^0 := \ker(\sD : \PP_d \ra \PP_{d-2});
\label{eq:PP0d}
\end{equation}
this is the degree-$d$ homogeneous part of
\begin{equation}
\PP^0 := \bigoplus_{d=0}^\infty \PP_d^0 = \ker(\sD : \PP \ra \PP).
\label{eq:PP0}
\end{equation}
For example, $\PP_0^0$ and $\PP_1^0$ are the spaces of
constant and linear functions respectively, of dimensions~$1$ and~$n$;
and a quadratic polynomial $P = \sum_{1\leq j \leq k \leq n} a_{jk} x_j x_k$
is harmonic if and only if $\sum_{j=1}^n a_{jj} = 0$,
because $\sD P$\/ is the constant polynomial $2\sum_{j=1}^n a_{jj}$.

It is well known, and we shall soon demonstrate, that
$\sD : \PP_d \ra \PP_{d-2}$ is surjective, whence
\begin{equation}
\dim(\PP_d^0) = \dim(\PP_d) - \dim(\PP_{d-2})
= \Binom{n+d-1}{d} - \Binom{n+d-3}{d}.
\label{eq:dimPP0}
\end{equation}
We shall use two further operators on $\rC^\infty(\R^n)$ and on
its subspace~$\PP$.  The first is
\begin{equation}
\sE := x \cdot \nabla = \sum_{j=1}^n x_j \frac{\del}{\del x_j}.
\label{eq:sE}
\end{equation}
Euler proved that
if $P \in \rC^\infty(\R^n)$ is homogeneous of degree~$d$\/
then $\sE P = d \cdot P$\/; in particular $\PP_d$ is the
\hbox{$d$-eigenspace} of~$\sE|_\PP$.
The second operator is multiplication by the norm:
\begin{equation}
\sF := \inprodl{x}{x} = \sum_{j=1}^n x_j^2 : \
P \mapsto \inprodl{x}{x} P.
\label{eq:sF}
\end{equation}
Clearly $\sF$ injects each $\PP_d$ into $\PP_{d+2}$.
Thus $\PP_d^0 = \ker(\sF\sD : \PP_d \ra \PP_d)$; that is,
$\PP_d^0$ is the zero eigenspace of the operator $\sF\sD$
on~$\PP_d$.
We next show that the other eigenspaces are $\sF^k \PP_{d-2k}^0$
for $k = 1, 2, \ldots, \lfloor d/2 \rfloor$,
and that $\PP_d$ is the direct sum of these eigenspaces, from which
the surjectivity of $\sD : \PP_d \ra \PP_{d-2}$
will follow as a corollary.

We begin with by finding the commutators of $\sD,\sE,\sF$.
Recall that the \defn{commutator}
of any two operators $A,B$\/ on some vector space is
\begin{equation*}
[A,B] = AB-BA = -[B,A].
\end{equation*}
For example,
$[x_j,x_k] = [\del / \del x_j, \del /  \del x_k] = 0$ for all $j,k$,
while $[\del / \del x_j, x_k] = \delta_{jk}$ (Kronecker delta).
Applying these formulas repeatedly, we obtain the commutation relations
\begin{equation}
[\sD,\sF] = 4 \sE + 2n,
\quad
[\sE,\sD] = -2 \sD,
\quad
[\sE,\sF] = 2 \sF.
\label{eq:DEFcom}
\end{equation}
This suggests the commutation relations
\begin{equation}
\label{eq:slcommrel}
\commut{\Xoper}{\Yoper}=\Hoper, \quad
\commut{\Hoper}{\Xoper}=2\Xoper,
\quad \commut{\Hoper}{\Yoper}=-2\Yoper
\end{equation}
satisfied by the standard basis
\begin{equation}
(\Xoper,\Hoper,\Yoper) = \bigl(
  \slmat{0}{1}{0}{0}, \,
  \slmat{1}{0}{0}{-1}, \,
  \slmat{0}{0}{1}{0}
\bigr)
\label{eq:sl2}
\end{equation}
of $\Slalg_2$.  Indeed (\ref{eq:DEFcom}) is tantamount to
an isomorphism of Lie algebras from $\Slalg_2$ to the span of
$\{\sD, \sE + \frac{n}{2}, \sF\}$ that takes $(\Xoper,\Hoper,\Yoper)$
to $(\frac1{2\varpi} \sD, -(\sE + \frac{n}{2}), -\frac\varpi2\sF)$
for some nonzero $\varpi$ (all choices of $\varpi$ are equivalent via
conjugation by diagonal matrices; later the choice $\varpi = 2\pi$
will be most natural for us).  Some steps in the following analysis
are familiar from the representation theory of~$\Slalg_2$, though here
only infinite-dimensional representations arise.

Now suppose $P \in \PP_d$ is in the $\lambda$-eigenspace of $\sF\sD$
for some~$\lambda$.  Then $\inprodl{x}{x} P = \sF P$\/ is in the
$(\lambda + 4d + 2n)$-eigenspace of $\sF\sD$ acting on $\PP_{d+2}$,
because
$$
\sF\sD\sF P = \sF(\sF\sD + [\sD,\sF]) P
= \sF(\sF\sD + 4\sE + 2n) P
= \sF(\lambda + 4d + 2n) P.
$$
By induction on $k=0,1,2,\ldots$ it follows that $\sF^k P$\/
is an eigenvector of $\sF\sD |_{\PP_{d+2k}}$ with eigenvalue
$$
\lambda + \sum_{j=0}^{k-1} 4(d+2j) + 2n
= \lambda + k \bigl( 4(d+k-1) + 2n \bigr).
$$
Replacing $d$\/ by $d-2k$\/ and taking $\lambda = 0$,
we see that if $P \in \PP_{d-2k}^0$ then $\sF^k P$\/
is an eigenvector of $\sF\sD |_{\PP_d}$ with eigenvalue
\begin{equation*}
\lambda_d(k) := k \bigl( 4(d-k-1) + 2n \bigr).
\label{eq:lambdak}
\end{equation*}
We next prove that this accounts for all the eigenspaces of
$\sF\sD|_{\PP_d}\0$.

\begin{lemma}
\label{lemma:easy}
Fix $d \geq 0$\,.  For integers $k,k'$ such that
$0 \leq k < k' \leq d/2$ we have $\lambda_d(k) < \lambda_d(k')$.
\end{lemma}

\begin{proof} By induction it is enough to check this for $k'=k+1$.
We compute
$$
\lambda_d(k+1) - \lambda_d(k) = 2n + 4(d-2k') \geq 2n > 0,
$$
as claimed.
\end{proof}

\begin{cor}
\label{cor:easy}
The sum of the subspaces $\sF^k \PP_{d-2k}^0$ of $\PP_d$
over $k=0,1,\ldots,\lfloor d/2 \rfloor$ is direct.
\end{cor}

\begin{proof}
By Lemma~\ref{lemma:easy}, the $\lambda_d(k)$ are strictly increasing,
and thus distinct.  Our claim follows because $\sF^k \PP_{d-2k}^0$
is a subspace of the $\lambda_d(k)$ eigenspace of $\sF\sD$.
\end{proof}

\begin{prop}
\label{prop:harmdecomp}
For $k=0,1,\ldots,\lfloor d/2\rfloor$,
let $\PP_d^k = \sF^k \PP_{d-2k}^0$. Then:
\begin{enumerate}
\item The map $\sD: \PP_d \ra \PP_{d-2}$ is surjective.
\item $\PP_d = \bigoplus_{k=0}^{\lfloor d/2\rfloor} \PP_d^k
= \PP_d^0 \oplus \sF \PP_{d-2}$, and
$\PP = \bigoplus_{k=0}^\infty \sF^k\PP^0$.
\item $\PP_d^k$ is the entire $\lambda_d(k)$ eigenspace of
$\sF\sD |_{\PP_d}$, and $\sF\sD |_{\PP_d}$
has no eigenvalues other than the $\lambda_d(k)$ for
$k=0,1,\ldots,\lfloor d/2\rfloor$.
\item $\dim (\PP_d^0) = \dim(\PP_d) - \dim(\PP_{d-2})$
as claimed in~(\ref{eq:dimPP0}).
\end{enumerate}
\end{prop}

\begin{proof}
The sum $\bigoplus_{k=0}^{\lfloor d/2\rfloor} \PP_d^k$
is direct by Corollary~\ref{cor:easy}.
We prove that it equals $\PP_d$ by comparing dimensions.  Since
$\sF$ is injective we have $\dim(\PP_d^k) = \dim(\PP_{d-2k}^0)$;
moreover
$$
\dim(\PP_{d-2k}^0) \geq \dim(\PP_{d-2k}) - \dim(\PP_{d-2k-2}),
$$
with equality if and only if
$\sD : \PP_{d-2k} \ra \PP_{d-2k-2}$ is surjectve, because
$\PP_{d-2k}^0$ is the kernel of $\sD : \PP_{d-2k} \ra \PP_{d-2k-2}$.
Hence $\dim \bigl( \bigoplus_{k=0}^{\lfloor d/2\rfloor} \PP_d^k \bigr)$
is
\begin{equation}
\sum_{k=0}^{\lfloor d/2\rfloor} \dim(\PP_d^k)
= \sum_{k=0}^{\lfloor d/2\rfloor} \dim(\PP_{d-2k}^0)
\geq \sum_{k=0}^{\lfloor d/2\rfloor}
  \bigl(\dim(\PP_{d-2k}) - \dim(\PP_{d-2k-2})\bigr),
\label{eq:dimbound}
\end{equation}
and the last sum telescopes to $\dim(\PP_d)$.
Thus equality holds termwise in the last step of (\ref{eq:dimbound}) and
$\dim \bigl(\bigoplus_{k=0}^{\lfloor d/2\rfloor} \PP_d^k\bigr)
= \dim (\PP_d)$.
The first of these proves part~(1) (using the $k=0$ term).
The second yields
\begin{equation}
\PP_d = \bigoplus_{k=0}^{\lfloor d/2\rfloor} \PP_d^k,
\label{eq:PPdecomp}
\end{equation}
as claimed in part~(2); taking the direct sum over~$d$\/ yields
$\PP = \bigoplus_{k=0}^\infty \sF^k\PP^0$, also claimed in part~(2).
To complete the proof of part~(2)
we compare the decompositions~(\ref{eq:PPdecomp})
of $\PP_d$ and $\PP_{d-2}$ and note that
$\PP_d^k = \sF \PP_{d-2}^{k-1}$ for each $k>0$.
Part~(3) follows because the decomposition~(\ref{eq:PPdecomp})
diagonalizes $\sF\sD |_{\PP_d}$.  Finally part~(4) is again
the equality of the $k=0$ terms in (\ref{eq:dimbound}).
\end{proof}

\subsubsection*{Remarks}
Part (2) of Proposition \ref{prop:harmdecomp} says in effect that
$\PP = \bigoplus_{d=0}^\infty \bigl( \PP_d^0 \otimes U_{\frac{n}{2}+d} \bigr)$,
where for any real $m>0$ we write $U_m$ for
the infinite-dimensional irreducible representation of~$\Slalg_2$
with basis $\{\Yoper^k v\}_{k=0}^\infty$ where
$\Xoper v = 0$ and $\Hoper v = -mv$.  These $U_m$ come from
representations in the ``discrete series'' of unitary representations
of the Lie group $\SL_2(\R)$ when $n$ is even
(see \cite[Chapter~IX] {Lang:SL2R}); when $n$ is odd,
they come from discrete-series representations of the ``metaplectic''
double cover of $\SL_2(\R)$ that do not descend to $\SL_2(\R)$.

It also follows from part~(2) that $\PP^0_d \cap \sF \PP_{d-2} = \{0\}$,
and thus that $\PP^0$ contains no nonzero multiple of $\inprodl{x}{x}$.
Proving this was set as problem B-5 on the 2005 Putnam exam,
which was the hardest of the 12 problems that year,
solved by only five of the top 200 scorers
(see \cite[p.736 and p.~741]{Putnam05}).
The solution printed in \cite[p.~742]{Putnam05} uses some of the
ingredients used here to prove Proposition~\ref{prop:harmdecomp}.

         \subsection{Weighted Theta Functions}\label{subsec:wtf_section}
           The functional equation~\eqref{eq:Thetadual} for theta functions of lattices
extends to theta functions weighted by a harmonic polynomial.
\begin{theorem}\label{thm:thetaP_func_eq}
Let $L$ be any lattice in~$\R^n$, and $P: \R^n \ra \C$
any harmonic polynomial of degree~$d$.  Then 
\begin{equation}
\Theta\0_{L^{\kern-.2ex*}\9\!,P}(e^{-2\pi t}) =
i^d \Vol(\R^n/L) t^{-(n/2)-d} \Theta\0_{L,P}(e^{-2\pi/t})
\label{eq:ThetaPdual}
\end{equation}
for all $t>0$.
\end{theorem}
By the Poisson summation formula, this will follow from the
following generalization of~\eqref{eq:Gaussian-hat}:
\begin{theorem}\label{thm:PGF}
Suppose that $t>0$ and $P: \R^n \ra \C$ is a harmonic polynomial on~$\R^n$
of degree~$d$.  Define a function $f: \R^n \ra \R$ by
\begin{equation}
f(x) = P(x) \, e^{-\pi \inprodl{x}{x} t}.
\label{eq:f=PG}
\end{equation}
Then the Fourier transform of~$f$\/ is
\begin{equation}
\hat f(y) = i^d  t^{-(\frac{n}{2}+d)} P(y) \, e^{-\pi \inprodl{y}{y} / t}.
\label{eq:f^=PG}
\end{equation}
\end{theorem}

\begin{proof}
For $t \in \C$\/ define an operator
\begin{equation}
\Goper{t} : \rC^\infty(\R^n) \ra \rC^\infty(\R^n), \quad
g \mapsto e^{-\pi t \inprodl{x}{x}} g
\label{eq:Gt}
\end{equation}
that multiplies every $\rC^\infty$ function by the Gaussian
$e^{-\pi t \inprodl{x}{x}}$;
these operators constitute a one-parameter group:
$\Goper{t} \Goper{t'} = \Goper{t+t'}$ for all $t,t'$.
We are then interested in $f = \Goper{t} P$\/ for $P \in \PP$\/
in the intersection of the kernel of $\sD$
with an eigenspace of~$\sE$.  If $P \in \PP_d$ then
\begin{equation*}
d \cdot f = \Goper{t} (d\cdot P) = \Goper{t} \sE P
= (\Goper{t} \sE \Goper{-t}) \Goper{t} P
= (\Goper{t} \sE \Goper{-t}) f,
\label{eq:conjE}
\end{equation*}
so $f$\/ is in the $d$-eigenspace of $\Goper{t} \sE \Goper{-t}$;
likewise $f \in \ker \Goper{t} \sD \Goper{-t}$.
Since our one-parameter group $\{\Goper{t}\}$ has infinitesimal generator
$-\pi\sF$, we expect that conjugation by $\Goper{t}$ will take $\sD,\sE$
to some linear combination of $\sD,\sE,\sF$.  Indeed we find the
following relations.\footnote{
  This is where it becomes natural to use $\varpi = 2\pi$ when choosing
  the images of the generators~\eqref{eq:sl2} of~$\Slalg_2$:
  conjugation by $\Goper{t}$ then takes $(\Xoper,\Hoper,\Yoper)$ to
  $(\Xoper-t\Hoper-t^2\Yoper, \Hoper+2t\Yoper, \Yoper)$; other choices
  would produce more complicated coefficients.
  }
\begin{lemma}[Conjugation of $\sD,\sE,\sF$ by $\Goper{t}$]\label{lemma:conjugation}
The operators $\Goper{t}$ commute with $\sF$, and we have
\begin{equation}
\Goper{t} \sE \Goper{-t} = \sE + 2 \pi t \sF,
\quad
\Goper{t} \sD \Goper{-t} = \sD + \pi t (4\sE + 2n) + (2 \pi t)^2 \sF.
\label{eq:DEconj}
\end{equation}
\end{lemma}%
\begin{proof} As with the commutation relations~\eqref{eq:DEFcom},
this comes down to an exercise in differential calculus.
Here we start from the fact that $\Goper{t}$ commutes with each $x_j$
while
$
\Goper{t} (\del / \del x_j) \Goper{-t} = 2 \pi t x_j + (\del / \del x_j)
$, whence the first formula in~\eqref{eq:DEconj} quickly follows,
while $\Goper{t} \sF = \sF \Goper{t}$ is immediate.
A somewhat longer computation establishes the second formula.
\end{proof}

\begin{cor}
\label{cor:lemma_conj}
The operators $\sD,\sE,\sF$ act on $\Goper{t} \PP$, and
the subspace $\Goper{t} \PP_d^0$ is the intersection of\/
$\ker(\sD + \pi t (4\sE + 2n) + (2 \pi t)^2 \sF)$ with the
$d$-eigenspace of\/ $\sE + 2 \pi t \sF$ in $\Goper{t} \PP$.
\end{cor}

We next relate the Fourier transform of a Schwartz function~$f$\/
with the Fourier transforms of its images under $\sD,\sE,\sF$.

\begin{lemma}[Conjugation of $\sD,\sE,\sF$ by the Fourier Transform]\label{lemma:fourconj}
Let $f : \R^n \ra \C$ be any Schwartz function.  Then:
\begin{enumerate}
\item For each $j=1,\ldots,n$,
the Fourier transform of $x_j f$\/ is $(2\pi i)^{-1} \del \hat f / \del y_j$,
and the Fourier transform of $\del f / \del x_j$ is $-2\pi i y_j \hat f$.
\item The Fourier transforms of $\sD f$, $(2\sE+n) f$, and $\sF f$
are respectively
$-(2\pi)^2 \sF \hat f$, $-(2\sE+n) \hat f$, and $-(2\pi)^{-2} \sD \hat f$.
\end{enumerate}
\end{lemma}

\begin{proof}
Again this is a calculus exercise, here with definite integrals.
The formula for the Fourier transform of $\del f / \del x_j$ is obtained
by integrating by parts with respect to~$x_j$.  The Fourier transform of
$x_j f$\/ can be obtained from this using Fourier inversion,
or directly by differentiation with respect to~$y_j$ of
the integral~\eqref{eq:fourier} that defines~$\hat f(y)$.
We then obtain part~(2) by iterating the formulas in part~(1) to find
the Fourier transform of $\del^2 f / \del x_j^2$,
$x_j \del f / \del x_j$, or $x_j^2 f$, and summing over~$j$.
The case of $\sE f$\/ can be explained by writing the operator $2\sE + n$
as $\sum_{j=1}^n \bigl(
  x_j (\del / \del x_j) + (\del / \del x_j) \circ x_j
\bigr)$.
\end{proof}

We use this to show that if $f \in \Goper{t} \PP$\/ then
$\hat f \in \Goper{1/t} \PP$, that is, that $\hat f$\/ is
{\em some}\/ polynomial multiplied by $e^{-\pi\inprodl{y}{y}/t}$.
More precisely:

\begin{prop}
\label{prop:fourier}
Let $t \in \C$ with $\Re(t) > 0$.
If $f = \Goper{t} P$\/ for some $P \in \PP_d$ then
$\hat f = \Goper{1/t} \Phat$\/ for some
$\Phat = \sum_{d'=0}^d \Phat_{d'}$ with each $\Phat_{d'} \in \PP_{d'}$
and $\Phat_d = i^d t^{-(\frac{n}{2}+d)} P$.
As before $t^{-(\frac{n}{2}+d)}$ denotes the $-(n+2d)$ power of
the principal square root of~$t$.
\end{prop}

\begin{proof}
We use induction on~$d$.  The base case $d=0$
is the fact that the Fourier transform of $e^{-\pi t \inprodl{x}{x}}$ is
$t^{-n/2} e^{-\pi\inprodl{y}{y}/t}$, which we showed already.  Suppose
we have established the claim for $P \in \PP_d$.  By linearity
and the fact that $\PP_{d+1}$ is spanned by its subspaces $x_j \PP_d$,
it is enough to prove the proposition with $P$\/ replaced by $x_j P$.
By the first part of Lemma~\ref{lemma:fourconj},
the Fourier transform of $\Goper{t} x_j P = x_j \Goper{t} P$\/ is
\begin{equation}
\frac1{2 \pi i} \, \frac{\del}{\del y_j} \bigr( \Goper{1/t} \Phat \bigr)
= \frac1{2 \pi i} \Goper{1/t} \Bigl(
   \frac{\del \Phat}{\del y_j} - \frac{2 \pi}{t} y_j \Phat 
  \Bigr).
\label{eq:Lemma_Fourier}
\end{equation}
By the inductive hypothesis $\Phat$\/ has degree~$d$\/
and leading part $\Phat_d = i^d t^{-(\frac{n}{2}+d)} P$.
Therefore the right-hand side of~\eqref{eq:Lemma_Fourier}
has degree $d+1$ and leading part
$$
\frac{-2 \pi t^{-1}}{2 \pi i} \Phat_d
= \frac{i}{t} \Phat_d = i^{d+1} t^{-(\frac{n}{2}+d+1)} y_j P.
$$
This completes the inductive step and the proof.
\end{proof}

To finish the proof of Theorem~\ref{thm:PGF}, suppose
$P \in \PP^0_d$ and $f(x) = P(x) \, e^{-\pi \inprodl{x}{x} t} = \Goper{t} P$.
By Corollary~\ref{cor:lemma_conj},
\begin{equation*}
(\sD + \pi t (4\sE + 2n) + (2 \pi t)^2 \sF) f = 0,
\qquad
(\sE + 2 \pi t \sF) f = d \cdot f.
\label{eq:harm_f}
\end{equation*}
Taking the Fourier transform and applying the second part of
Lemma~\ref{lemma:fourconj}, we deduce
\begin{equation*}
(-(2\pi)^2 \sF - \pi t (4\sE + 2n) - t^2 \sD) \hat f = 0,
\quad\
-\Bigl(\sE + n + \frac{t}{2\pi} \sD\Bigr) \hat f = d \cdot \hat f.
\label{eq:harm_fhat}
\end{equation*}
Eliminating $\sD \hat f$, we find
$d \cdot \hat f = (\sE + \frac{2\pi}{t} \sF) \hat f$\/;
that is, $\hat f$\/ is in the $d$-eigenspace of $\sE + 2 \pi t^{-1} \sF$.
By Proposition~\ref{prop:fourier}, we know that
$\hat f = \Goper{1/t} \Phat$\/ for some $\Phat \in \PP$.
By Lemma~\ref{lemma:conjugation}, then,
$\Phat$\/ is in the \hbox{$d$-eigenspace} of~$\sE$;
that is, $\Phat \in \PP_d$.  By Proposition~\ref{prop:fourier},
we conclude that $\Phat = i^d t^{-(\frac{n}{2}+d)} P$.
\end{proof}

We have now proven the functional equation~\eqref{eq:ThetaPdual}
for weighted theta functions $\Theta\0_{L,P}$
(Theorem~\ref{thm:thetaP_func_eq}).
This identity is trivial when $d = \deg(P)$ is odd, because then
$\Theta\0_{L,P}$ is identically zero (by cancellation of the
$v$ and $-v$ terms), but it gives new information when $d$\/ is 
even and positive.

Again we consider the special case of a Type~II lattice.
Generalizing \eqref{eq:theta}, we define
\begin{equation}
\theta\0_{L,P}(\tau) := \Theta\0_{L,P}(e^{2\pi i \tau})
     = \sum_{v\in L} P(v) e^{\pi\inprodl{v}{v} i \tau}
\label{eq:thetaP}
\end{equation}
for $\tau \in \upperH$.
Then $\theta\0_{L,P}(\tau) = \theta\0_{L,P}(\tau + 1)$,
and Theorem~\ref{thm:thetaP_func_eq} gives
$\theta\0_{L,P}(\tau) = t^{-(\frac{n}{2}+d)}\9 \theta\0_L(-1/\tau)$,
with the factor $i^d$ absorbed by the change of variable $\tau = it$
because $d$\/ is even.  It follows as before that
\begin{equation}
\theta\0_{L,P}(\tau) = (c\tau+d)^{-(\frac{n}{2}+d)}\9 \,
   \theta\0_{L,P}\left(\frac{a\tau+b}{c\tau+d}\right)
\label{eq:thetaP_form}
\end{equation}
for all $\slmat{a}{b}{c}{d} \in \SL_2(\Z)$, so $\theta\0_{L,P}$
is a modular form of weight $\frac{n}{2} + d$\/ for $\SL_2(\Z)$.
Hence $\theta\0_{L,P}$ is a polynomial in $\Eis_4$ and the
\hbox{weight-$6$} Eisenstein series
$$
\Eis_6 = 1 - 504 \sum_{n=1}^\infty \frac{n^5 q^n}{1-q^n}
= 1 - 504 q - 16632 q^2 - 122976 q^3 - \cdots.
$$
Moreover $\theta\0_{L,P}$ is a cusp form once $d>0$:
the constant coefficient is $P(0)$, which vanishes for
nonconstant homogeneous~$P$.  Hence once $d$ is positive
the polynomial giving $\theta\0_{L,P}$ in terms of $\Eis_4$ and $\Eis_6$
is divisible by $\Delta = 12^{-3} (\Eis_4^3 - \Eis_6^2)$.
(See again \cite[Chapter~VII]{Serre:course}.)

In particular $\theta\0_{L,P} = 0$ when
$\frac{n}{2} + d \in \{2, 4, 6, 8, 10, 14\}$ because in those weights
the only cusp form is the zero form.\footnote{
  In this setting $\frac{n}{2} + d$\/ cannot be as small as~$2$
  because $n \geq 8$, but the possibility of weight~$2$ arises
  in the proof of Lemma~\ref{lemma:thetaLP=0}.
  }
Likewise we have the following observation.
\begin{lemma}
\label{lemma:thetaLP=0}
Suppose $n=8m$ and $L \subset \R^n$ is an extremal lattice.
Then $\theta_{L,P} = 0$ for every nonconstant harmonic polynomial
$P$ on $\R^n$ whose degree $d$ satisfies
$4(m - 3 \lfloor m/3 \rfloor) + d \in \{2, 4, 6, 8, 10, 14\}$.
If $L \subset \R^n$ is a Type~II lattice of minimal norm $n/12$
then $\theta_{L,P} = 0$ for every harmonic polynomial
$P$ on $\R^n$ of degree~$2$.
\end{lemma}
\begin{proof}
We saw already that $\theta_{L,P}$ is a cusp form.  If $L$ is
extremal, the $q^k$ coefficient of $\theta_{L,P}$ vanishes for each
$k \leq \lfloor m/3 \rfloor$.
Hence $\Delta^{\!-\lfloor m/3 \rfloor} \theta_{L,P}$
is a cusp form of weight $4(m - 3 \lfloor m/3 \rfloor) + d$,
and thus vanishes when
$4(m - 3 \lfloor m/3 \rfloor) + d \in \{2, 4, 6, 8, 10, 14\}$.
Likewise if $L$ has minimal norm $n/12$ and $P$ is a quadratic
harmonic polynomial then $\Delta^{\kern-.05ex1 - (n/24)} \theta_{L,P}$
is a cusp form of weight~$14$, so again $\theta_{L,P} = 0$.
\end{proof}

If $L$ is extremal then
Lemma~\ref{lemma:thetaLP=0} applies to $6$, $4$, or $2$ values of~$d$\/
for $n \equiv 0$, $8$, or $16 \bmod 24$ respectively.
We exploit these vanishing results in the next section.

         \subsection{Spherical $t$-Designs, Extremal Type~II Lattices, and the Venkov condition on Niemeier Lattices}
           For real $\nu>0$ let $A_\nu : \rC^\infty(\R^n) \ra \C$ be the functional
that takes any function to its average on the sphere
$\Sigma_\nu = \{ x \in \R^n \setsep \inprodl{x}{x} = \nu \}$ with respect to
the probability measure on~$\Sigma_\nu$ invariant under the orthogonal group.
For any positive integer~$t$, a (possibly empty\footnote{
  With this definition $\emptyset$ is a $t$-design for all~$t$.
  For most applications only nonempty designs are of interest;
  for instance it is only when $D$\/ is nonempty that we can divide
  both sides of \eqref{eq:tdesign_def} by $|D|$ to get
  the equivalent condition that the average of any polynomial
  of degree at most~$t$\/ over~$\Sigma_\nu$
  can be computed by averaging it over~$|D|$.
  But we allow empty designs here, and also later in the
  coding-theoretic setting, because this simplifies the statements of
  the results relating lattices with spherical designs.
  }) finite set $D \subset \R^n$ of nonzero vectors of equal norm~$\nu$
is said to be a \defn{(spherical) $t$-design} if and only if
\begin{equation}
\sum_{v \in D} P(v) = \left|D\right| \cdot A_\nu(P)
\label{eq:tdesign_def}
\end{equation}
for all $P \in \PP$ with $\deg P \leq t$.\footnote{
  See \cite{Delsarte:Hahn} for explanation of the term ``$r$-design''
  for this property.  For $D \neq \emptyset$, the $t$-design property
  is one way to make precise the idea that $D$\/ is ``well distributed''
  in~$\Sigma_\nu$, and better distributed as $t$ grows.  One application,
  and the original one according to \cite[pp.~89-90]{SPLAG}, is
  numerical integration on~$\Sigma_\nu$, using the right-hand side
  of~\eqref{eq:tdesign_def} as an approximation to the left-hand side
  even when $P$\/ is not polynomial but smooth enough to be
  well approximated by polynomials.
  }
By linearity it is enough to check this condition for $P \in \PP_d$
for each $d \leq t$, and may assume $d>0$ because
in the case $d=0$ of a constant polynomial the condition
\eqref{eq:tdesign_def} is satisfied automatically.
We next prove that it is enough to check \eqref{eq:tdesign_def}
for \emph{harmonic}\/ polynomials of positive degree.
We begin by showing that all such polynomials are in $\ker(A_\nu)$.

\begin{lemma}
\label{lemma:A(sph)=0}
If $P$ is a nonconstant harmonic polynomial then $A_\nu(P) = 0$.
\end{lemma}

\begin{proof}
Choose any $s>0$.  Since $P$\/ is homogeneous, $A_\nu(P)$ is
a positive multiple of the integral of $\Goper{s} P$ over all of~$\R^n$.
But this integral is the value of the Fourier transform of $\Goper{s} P$\/
at the origin.
By Theorem~\ref{thm:PGF} this value is some multiple of $P(0)$.
Since $d>0$ we have $P(0)=0$, so $A_\nu(P)=0$ as claimed.
\end{proof}

Thus our design criterion can be stated as follows.

\begin{lemma}
\label{lemma:design_sph}
A finite subset $D \subset \Sigma_\nu$ is a $t$-design if and only if
$\sum_{v \in D} P(v) = 0$ for all nonconstant harmonic polynomials $P$
of degree at most~$t$.
\end{lemma}

\begin{proof}
The ``only if'' direction is immediate from Lemma~\ref{lemma:A(sph)=0}.
We prove the ``if'' implication.
By the second part of Proposition~\ref{prop:harmdecomp} any polynomial
of degree $d \leq t$ can be written as
$\sum_{k=0}^{\lfloor d/2 \rfloor} \sF^k P_k$ with each $P_k$
harmonic of degree $d-2k$.  It is thus enough to check
\eqref{eq:tdesign_def} for each $\sF^k P_k$.
But by hypothesis, \eqref{eq:tdesign_def} holds for each $P_k$
(including $P_{d/2}$ if $d$ is even, because then $P_k$ is constant).
Since the restriction of each $\sF^k P_k$ to $\Sigma_\nu$ is $\nu^k P_k$,
it follows that \eqref{eq:tdesign_def} holds for $\sF^k P_k$ as well,
and we are done.
\end{proof}

Combining this with Lemma~\ref{lemma:thetaLP=0}
yields the following theorem of Venkov~\cite{Venkov:reseaux},
which asserts that in an extremal or nearly extremal Type~II
lattice the vectors of each nonzero norm form a spherical design.

\begin{theorem}
\label{thm:sph_t_des}
Let $L \subset \R^n$ be a Type~II lattice with minimal norm $2k$.
Assume $r := 24k - n$ is nonnegative.  Set $t=3$ if $r=0$ and
$t = (r/2) - 1$ if $r>0$.  Then $L \cap \Sigma_\nu$ is a
\hbox{$t$-design} for every $\nu>0$.
\end{theorem}

\begin{proof}
Because $L \cap \Sigma_\nu$ is centrally symmetric,
we need only check the criterion of Lemma~\ref{lemma:design_sph}
for $P$\/ of even degree.  For such~$P$, Lemma~\ref{lemma:thetaLP=0}
applies, so $\theta_{L,P} = 0$.  The criterion $A_\nu(P) = 0$
then holds because $A_\nu(P)$ is a coefficient of $\theta_{L,P}$.
\end{proof}

\subsubsection*{Remarks}
In general $L \cap \Sigma_\nu$ need not be a \hbox{$(t+1)$-design}:
there will be lattice norms~$\nu$ and harmonic polynomials $P$\/
of degree $t+1$ whose sum over $L \cap \Sigma_\nu$ is nonzero.
However, when $r>0$ it will be true that the sum over
$L \cap \Sigma_\nu$ of any harmonic polynomial of degree $t+3$
vanishes, because there are no nonzero cusp forms of weight~$14$.
Thus each $L \cap \Sigma_\nu$ is what Venkov~\cite{Venkov:reseaux}
called a ``\defn{$t\frac12$-design}'': a finite subset
$D \subset \Sigma_\nu$ such that $\sum_{v \in D} P(v) = 0$
for all $P \in \PP_d^0$ with either $d \leq t$ or $d = t+3$.

The fact that in each case $L \cap \Sigma_\nu$ is a \hbox{$2$-design}
already lets us deduce that if $L \cap \Sigma_\nu$ is nonempty
then it spans $\R^n$ as a vector space.
Indeed if $L \cap \Sigma_\nu$ does not span $\R^n$ then it is contained
in a hyperplane $\{ x \in \R^n \setsep \inprodl{x}{\xfixed} = 0 \}$
for some nonzero $\xfixed \in \R^n$; then we can take
$P(x) = \inprodl{x}{\xfixed}^2$ in~\eqref{eq:tdesign_def} and
observe that each of the terms $P(v)$ in the left-hand side vanishes,
while the factor $A_\nu(P)$ of the right-hand side is strictly positive,
so the remaining factor $|D|$ must vanish,
making $L \cap \Sigma_\nu = \emptyset$ as claimed.

More precise results can often be obtained when
$\nu$ equals or slightly exceeds the minimal norm,
because then any two vectors in $L \cap \Sigma_\nu$ must have
integer inner product, and only a few integers can arise,
making the condition that $L \cap \Sigma_\nu$ be a $t$-design or a
\hbox{$t\frac12$-design} particularly stringent.
We give three examples: configuration results for extremal Type~II
lattices in several dimensions, including multiples of~$24$ up to $96$,
showing that such lattices are generated by their minimal vectors;
Venkov's simplification of Niemeier's
classification of Type~II lattices in $\R^{24}$;
and a novel proof of the uniqueness of the $E_8$ lattice.

\subsubsection*{Configuration results for extremal Type~II lattices}
While a nonempty shell $L \cap \Sigma_\nu$ in an extermal lattice~$L$
must generate $\R^n$ as a vector space, it need not generate $L$ over~$\Z$:
already $(L,\nu) = (D_{16}^+,2)$ is a counterexample, since the
minimal nonzero vectors of $D_{16}^+$ generate only the
\hbox{index-$2$} sublattice $D_{16}$.  Still, for some $n$
it can be proved that every extremal lattice is generated by
its vectors of minimal norm~$2k$.  Let $L_0$ be the sublattice of $L$
generated by the minimal vectors, and assume $[L:L_0]>1$.
Then there are nonlattice vectors $\dot v \in L_0^*$, and
$\inprodl{v}{\dot v} \in \Z$ for all $v \in L \cap \Sigma_{2k}$.
If $\dot v$ has minimal norm in its coset $\bmod\,L$ then 
$\left|\inprodl{v}{\dot v}\right| \leq k$ for all such~$v$.
This together with the \hbox{$t$-design} or \hbox{$t\frac12$-design}
condition on $L \cap \Sigma_{2k}$ yields a contradiction for
several values of~$n$, proving that $L_0 = L$ for each of those~$n$.
(See \cite{Venkov:32}, \cite{Ozeki:32}, \cite{Ozeki:48},
\cite{Kominers:56+72+96}, and \cite{Elkies:40r}.)

\subsubsection*{Niemeier lattices}
Suppose $L$ is a Type~II lattice in $\R^{24}$.
Then the hypothesis of Theorem~\ref{thm:sph_t_des} is satisfied with
$r=0$ or $r=24$.  In either case we find in particular that
$L \cap \Sigma_2$ is a \hbox{$2$-design}.  But the vectors of norm~$2$
in any even lattice constitute a root system.  Venkov~\cite{Venkov:24},
used the condition that this root system be a \hbox{$2$-design}
to show {\em a priori} that it must be among the $24$ root systems
that arise for the Niemeier lattices,
and thus to considerably streamline the classification of
Type~II lattices in $\R^{24}$.

\subsubsection*{The uniqueness of $E_8$}
Finally, let $n=8$ and let $L \subset \R^8$ be any Type~II lattice.
Then $\theta_L = \Eis_4 = 1 + 240q + 2160q^2 + \cdots$,
and $L$ is automatically extremal, so in particular
$L \cap \Sigma_2$ is a \hbox{$7$-design} of size~$240$.
We shall use these facts to prove that $L \cong E_8$.
There are $2160$ vectors of norm~$4$ in~$L$;
choose one, and call it $\xfixed$.
Let $D$\/ be the \hbox{$7$-design} $L \cap \Sigma_2$.
For $j \in \Z$ let $N_j$ be the number of vectors
$x \in D$ such that $\inprodl{x}{\xfixed} = j$.
If $N_j \neq 0$, then $|j| \leq \sqrt{8}$ (by Cauchy\thmnamesep Schwarz)
and $j \in \Z$ (because $\inprodl{v}{v'} \in \Z$ for all $v,v'\in L$);
hence $j \in \{-2, -1, 0, 1, 2\}$.  Therefore
\begin{equation}
\label{eq:sum=240}
\sum_{j=-2}^2 N_j = |D| = 240.
\end{equation}
Since $D$\/ is centrally symmetric, $N_{-j} = N_j$ for each~$j$.
Finally, since $D$\/ is a \hbox{$7$-design},
\eqref{eq:tdesign_def} holds with $P(x) = \inprodl{x}{\xfixed}^d$
for each positive integer $d \leq 7$.  This is automatic for $d$\/ odd,
but for $d=2,4,6$ we get linear equations in $N_0, N_1, N_2$,
and already the $d=2$ and $d=4$ equations
together with (\ref{eq:sum=240}) let us solve for the $N_j$.  We find
\begin{equation}
\label{eq:nj_solve}
(N_{-2}, N_{-1}, N_0, N_1, N_2) = (14, 64, 84, 64, 14).
\end{equation}
(See the Remarks at the end of this section
for the evaluation of the functional $A_\nu$ on
even powers of $\inprodl{x}{\xfixed}$.)  In particular there are
$14$ vectors in~$D$, call them $v_i$ for $1 \leq i \leq 14$,
whose inner product with~$\xfixed$ is $2$.

For each $i$ we obtain a lattice vector $x_i = 2v_i - \xfixed$ that is
orthogonal to $\xfixed$ and satisfies $\inprodl{x_i}{x_i} = 4$ and
$x_i \equiv \xfixed \bmod 2L$.  For any $i$ and $i'$ we have
\begin{align*}
\inprodl{x_i}{x_{i'}} = \inprodl{2v_i-\xfixed}{2v_{i'}-\xfixed} &=
4\inprodl{v_i}{v_{i'}}
 - 2\inprodl{v_i}{\xfixed}
 - 2\inprodl{\xfixed}{v_{i'}}
 + \inprodl{\xfixed}{\xfixed}
\nonumber
\\
& = 4\inprodl{v_i}{v_{i'}} - 4 - 4 + 4
\nonumber
\\
& = 4\inprodl{v_i}{v_{i'}} - 4 \
\nonumber
\\
& \equiv 0 \bmod 4.
\label{eq:4|inprodl}
\end{align*}
Thus the vectors $x_i$ for $1 \leq i \leq 14$,
together with $\xfixed$ and $-\xfixed$, are $16$ vectors of norm~$4$,
any two of which are equal, opposite, or orthogonal.
Hence the $x_i$ together with $\pm \xfixed$ are the minimal vectors
of an isometric copy of $2\Z^8$ in~$L$.
Moreover $L$ also contains $v_i = (\xfixed + x_i)/2$,
and thus contains the \hbox{$\Z$-span} of $\xfixed$ and the $v_i$,
which is isometric with $D_8$.  But $L$ is self-dual, so
$D_8^* \subset L \subset D_8\0$.
Of the three lattices satisfying this condition,
one is $\Z^8$, which is of Type~I,
and the other two are isomorphic with $E_8$.
Therefore $L \cong E_8$, as claimed.

\subsubsection*{Remarks} A related proof, parallel to the beginning of
Conway's proof~\cite{Conway:Leech} of the uniqueness of the Leech lattice,
starts from the observation that each of the $2^8$ cosets of $2L$ in~$L$
intersects $\{ v \in L \setsep \inprodl{v}{v} \leq 4\}$ in either
$\{0\}$, a pair of minimal vectors, or at most $8$ orthogonal pairs of 
vectors of norm~$4$.  This accounts for at least
$1 + 240/2 + 2160/16 = 256 = 2^8$ cosets.  Hence equality holds
throughout, and any of the nonzero cosets that does not meet $\Sigma_2$
gives us a copy of $D_8$ in~$L$.  This approach uses only the
modularity of $\theta_L$, not of the more general $\theta_{L,P}$,
though it applies in fewer cases.
Either technique also yields the number of automorphisms of~$E_8$:
there are $2160$ choices of~$\xfixed$, and $2^7 7!$ automorphisms of
$D_8$ that fix $\xfixed$, half of which send $E_8$ to itself,
so $\left|\Aut(E_8)\right| = 2160 \cdot 2^6 7! = 696729600$.

For even $d \geq 0$, and a given vector $\xfixed$ of norm $\nufixed > 0$,
the average over $\Sigma_\nu$ of $\inprodl{x}{\xfixed}^d$ is computed
as a quotient of Beta integrals.  We find that if
$P(x) = \inprodl{x}{\xfixed}^d$ then
\begin{equation}
\label{eq:Beta}
A_\nu(P) = (\nu\nufixed)^{d/2} \frac
  {\int_0^1 u^d (1-u^2)^{(n-3)/2} \, du}
  {\int_0^1 (1-u^2)^{(n-3)/2} \, du}
 = (\nu\nufixed)^{d/2}
   \frac{\Beta\bigl((d+1)/2, (n-1)/2\bigr)}{\Beta\bigl(1/2, (n-1)/2\bigr)},
\end{equation}
where $u$ is the normalized projection
$(\nu\nufixed)^{-1/2} |\inprodl{x}{\xfixed}|$.  Thus
\begin{equation}
\label{eq:Betaprod}
A_\nu(P) = (\nu\nufixed)^{d/2} \,
 \frac1n \, \frac3{n+2} \, \frac5{n+4} \cdots \frac{d-1}{n+d-2} \,.
\end{equation}
In our case $\nu\nufixed = 2 \cdot 4 = 8$, so
$A_\nu(P) = 1$, $12/5$, $8$ for $d=2,4,6$.

Alternatively we could have applied Lemma~\ref{lemma:design_sph}
to the \defn{zonal spherical harmonics}, which are harmonic polynomials
that depend only on $\inprodl{x}{\xfixed}$.  For each degree~$d$\/
there is a one-dimensional space of zonal spherical harmonics,
proportional to a Gegenbauer orthogonal polynomial $C_m^{((n-2)/2)}(u)$
with $u = (\nu\nufixed)^{-1/2} \inprodl{x}{\xfixed}$.  This is
equivalent to using \eqref{eq:Beta} and \eqref{eq:Betaprod} for
\hbox{$t$-designs}, but for a \hbox{$t\frac12$-design} we need the
zonal spherical harmonics to exploit the vanishing of
$\sum_{v\in D} P(v)$ for $P \in \PP^0_{t+3}$.
This, too, has an analogue in the setting of discrete harmonic
polynomials, as in the proof of Theorem~\ref{thm:48+72c}
at the end of this paper.

   \section{Weight Enumerators of Binary Linear Codes}
            % this is the hadamard_section
\label{sec:hadamard_section}

\subsection{Coding-Theoretic Preliminaries}

By a \defn{(binary linear) code of length $n$} we mean a vector subspace
of the $\F_2$-vector space $\F_2^n$.  In this context, vectors of length~$n$
over~$\F_2$ are often called (binary) ``words'' of length~$n$.
The \defn{(Hamming) weight} of a word $w \in \F_2^n$,
denoted by $\wt(w)$, is the number of nonzero coordinates of~$w$,
and the \defn{(Hamming) distance} between two words $w,w' \in \F_2^n$
is $\wt(w'-w)$.
We denote by $\inprodc\cdot\cdot$ the usual bilinear pairing on~$\F_2^n$,
defined by $\inprodc{v}{w} = \sum_{j=1}^n v_j w_j$.  For a linear code
$C \subseteq \F_2^n$, the \defn{dual code} is the annihilator $C^\dualcode$
of~$C$\/ with respect to this pairing; thus
$\dim(C) + \dim(C^\dualcode) = n$ and $C^{\dualcode\dualcode} = C$\/
for every linear code $C \subseteq \F_2^n$.

If $C = C^\dualcode$ then $C$\/ is \defn{self-dual}.  Then
$\inprodc{c}{c'} = 0$ for all $c,c' \in C$, and in particular $\wt(c)$
is even for all $c \in C$\/ because $0 = \inprodc{c}{c}$ is the reduction of
$\wt(c)$ mod~$2$.  The map $\wt: C \ra \Z$ then reduces mod~$4$ to a
group homomorphism $C \ra 2\Z/4\Z$.  The code~$C$\/ is said to be
\defn{doubly even} or \defn{of Type~II} if this homomorphism is
trivial, that is, if $\inprodc{c}{c} \in 4\Z$ for all $c \in C$;
otherwise $C$\/ is said to be \defn{singly even} or \defn{of Type I}.
This notation reflects the analogy between binary linear codes and lattices.
It also respects the following construction (``Construction~A'' of
\cite{LeechSloane}; see also \cite[pp.~182--183]{SPLAG})
that associates a lattice $L_C \subset \R^n$
to any linear code $C \subseteq \F_2^n$:
\begin{equation}
\label{eq:construction_A}
L_C := \{ 2^{-1/2} v \setsep v \in \Z^n, \; v \bmod 2 \in C \} .
\end{equation}
Indeed $L_C^* = L_{C^\dualcode}\0$, so $L_C$ is self-dual if and only if
$C$\/ is, in which case $L_C$ is of Type~I or Type~II according as
$C$\/ is of Type~I or Type~II, respectively.

\subsubsection*{Examples}
If $C = C^\dualcode$ then $\dim(C) = n/2$, so $n$ is even.
For each positive even integer~$n$ there is a Type~I code of length~$n$
consisting of all $c$ such that $c_{2j-1} = c_{2j}$ for each $j \leq n/2$.
This is the unique Type~I code for $n=2$,
and is unique up to isomorphism (i.e., up to coordinate permutation)
for $n \leq 8$, but not unique for
any $n \geq 10$; and as with lattices the number of isomorphism classes
grows rapidly with~$n$.

If $\F_2^n$ contains a Type~II code then $n \equiv 0 \bmod 8$.
(This follows via Construction~A from the corresponding theorem for
lattices, but can also be proven directly.\footnote{
  Suppose $C$\/ is a self-dual code of length $n$.
  Then $C$\/ contains the \hbox{all-1s} vector $\onevec$, because
  $\inprodc{v}{v} = \inprodc{v}{\onevec}$ for all $v \in \F_2^n$, so
  $C \subseteq C^\dualcode$ implies $\onevec \in C^\dualcode$.
  Thus $C$\/ descends to a vector space of dimension $(n/2)-1$ in
  $\Vee := \{0,\onevec\}^\dualcode / \{0,\onevec\}$.  Since $2 \mid n$,
  the perfect pairing $\inprodc\cdot\cdot$ descends to a 
  perfect pairing on~$\Vee$, so a self-dual code is tantamount to
  a maximal isotropic subspace of~$\Vee$ relative to this pairing.
  If $4 \mid n$ then the map $\{0,\onevec\}^\dualcode \ra \F_2$,
  $v \mapsto (\wt(c)/2) \bmod 2$ descends to a quadratic form
  $\kwad: \Vee \ra \F_2$ consistent with that pairing.
  A Type~II code is then
  a self-dual code~$C$\/ that is totally isotropic relative to~$\kwad$.
  Such $C$\/ exists if and only if $(\Vee,\kwad)$ has Arf invariant zero.
  But the Arf invariant is $0$ or $1$ according as
  $\{v \in \Vee \colon \kwad(v) = 0\}$ has size
  $2^{n-3} + 2^{(n/2)-2}$ or $2^{n-3} - 2^{(n/2)-2}$.
  But this count is
  $(1/2) \sum_{j=0}^{n/4} {n \choose 4j}
  = (1/8) \sum_{\mu^4 = 1} (1+\mu)^n = 2^{n-3} + (1/4) \Re(1+i)^n$,
  so the result follows from the observation that $(1+i)^4 = -4$.
  })
An example is the \defn{extended Hamming code}
in $\F_2^8$: if we identify $\F_2^8$ with the space of $\F_2$-valued
functions on $\F_2^3$, the Hamming code can be constructed as the
subspace of affine-linear functions on~$\F_2^3$.  The extended
Hamming code is the unique Type~II code of length~$8$;
there are two such codes of length~$16$, nine of length~$24$,
and a rapidly growing number as $n \ra \infty$
through multiples of~$8$.

\subsection{Discrete Poisson Summation}
We define the \defn{discrete Fourier transform}
(or \defn{Hadamard transform})~$\htr{f}$ of a function
$f : \F_2^n \ra \C$ as the function on $\F_2^n$ given by
\begin{equation}
\label{eq:dft}
\htr{f}(u)=\sum_{v\in \F_2^n}(-1)^{\inprodc{u}{v}}f(v).
\end{equation}

We review the \earlyterm{discrete Poisson summation formula}, a discrete
analog of the Poisson summation formula for lattices
(Theorem~\ref{thm:psum}).  Like its lattice analog,
the discrete Poisson summation formula relates the sums of
a function to the sums of the function's discrete Fourier transform.
 Here, however, instead of considering the sums of the function and
its Fourier transform over a lattice $L\subset \R^n$ and its dual
$L^\duallat$, we consider the sums of the function and its discrete
Fourier transform over a linear code $C\subset\F_2^n$ and
over $C^\dualcode$, the dual code of $C$.

\begin{theorem}[Discrete Poisson Summation Formula]\label{thm:dpsum}
Let $C\subset \F_2^n$ be a binary linear code of length~$n$, and
let $f$ be a function from $\F_2^n$ to~$\C$.  Then
\begin{equation}\label{eq:dpsum}
\sum_{c\in C}f(c)=\frac{1}{|C^\dualcode|}\sum_{c'\in C^\dualcode}\htr{f}(c').
\end{equation}
\end{theorem}
We briefly recount the standard proof of Theorem~\ref{thm:dpsum},
which is the one presented in~\cite[p.~127]{MacWilliamsSloane:theory}.

\begin{proof}[Proof of Theorem~\ref{thm:dpsum}]
By expanding the sum in the right-hand side of \eqref{eq:dpsum}
and rearranging the order of summation, we obtain
\begin{equation}\label{eq:indpsum}
\sum_{c'\in C^\dualcode}\htr{f}(c')
 = \sum_{c'\in C^\dualcode}\sum_{v\in \F_2^n}(-1)^{\inprodc{c'}{v}}f(v)
 = \sum_{v\in \F_2^n}f(v)\sum_{c'\in C^\dualcode}
  (-1)^{\inprodc{c'}{v}}.
\end{equation}
Now, whenever $v\in C\subset \F_2^n$ and $c'\in C^\dualcode$, we
have $\inprodc{c'}{v}=0$ by the definition of $C^\dualcode$.  It
follows that the inner sum in \eqref{eq:indpsum} equals $|C^\dualcode|$
whenever $v\in C$.  Furthermore, when $v\not\in C$, the inner sum
of \eqref{eq:indpsum} vanishes.\footnote{
  In this case, $\inprodc{c'}{v}$ takes the values $0$ and $1$
  equally often (see \cite[p.~127]{MacWilliamsSloane:theory}).
  (This statement is just an instance of the well-known fact that the
  sum of a nontrivial character on a finite commutative group vanishes.)
  We could also adapt the technique we used in proving Theorem~\ref{thm:psum},
  obtaining discrete Poisson summation via the discrete Fourier
  expansion of the function $z \mapsto \sum_{c \in C} f(c+z)$.
  }
The result then follows immediately.
\end{proof}

         \subsection{The MacWilliams Identity and Gleason's Theorem}
            % this is the we_section
\label{sec:we}
In this section, we recall two classical results from coding theory
which are closely related to the theory of lattices.  The first of
these results, the MacWilliams identity (Theorem~\ref{thm:MacWilliams},
below), expresses the weight enumerator of $C^\dualcode$ in terms of the
weight enumerator of~$C$.  The second result (Theorem~\ref{thm:Gleason},
below) is a famous theorem originally due to Gleason~\cite{Gleason:gleason},
which shows that the weight enumerators of Type~II codes can be
expressed in terms of two particular weight enumerators.

\begin{theorem}[{MacWilliams Identity (\cite{MacWilliams:identity}\thmmulticitesep\cite[p.~78]{SPLAG}\thmmulticitesep\cite[p.~74]{Ebeling:lattices}\thmmulticitesep\cite[p.~126]{MacWilliamsSloane:theory})}]\label{thm:MacWilliams}
For any  binary linear code $C$ of length $n$, we have
\begin{equation}
W_{C}(x,y)=\frac{1}{|C^\dualcode|}W_{C^\dualcode}(x+y,x-y).
\label{eq:McW_Id}
\end{equation}
\end{theorem}
\begin{proof}
Define a function $f: \F_2^n \ra \C$ by $f(v)=x^{n-\wt(v)}y^{\wt(v)}$.
Then
$$
\htr{f}(u)=(x+y)^{n-\wt(u)}(x-y)^{\wt(u)}.
$$
Theorem~\ref{thm:MacWilliams} therefore follows directly from the
discrete Poisson summation formula (Theorem~\ref{thm:dpsum}).
\end{proof}

\begin{theorem}[{Gleason's Theorem (\cite{Gleason:gleason}\thmmulticitesep\cite{Sloane:invariants}\thmmulticitesep\cite[p.~192]{SPLAG}\thmmulticitesep\cite[p.~75]{Ebeling:lattices})}]\label{thm:Gleason}
For any Type~II code~$C$, the weight enumerator $W_C(x,y)$ is
a polynomial in
\begin{equation}
\gleasonphi\defeq W_{e_8}(x,y)
= x^8+14x^4y^4+y^8\quad\text{and}\quad\gleasonxi\defeq
x^4y^4(x^4-y^4)^4.
\label{eq:gleason_greek}
\end{equation}
\end{theorem}

\begin{proof}
Since $C$\/ is of Type~II, the exponent of~$y$ in each monomial
$x^{n-\wt(v)}y^{\wt(v)}$ is a multiple of~$4$.  Thus each monomial
is invariant under the substitution of $iy$ for $y$,
whence the sum $W_C(x,y)$ of these monomials also satisfies
the identity $W_C(x,y) = W_C(x,iy)$.
Since $C = C^\dualcode$,
we also have an identity
\begin{align}
W_C(x,y) &= \frac{1}{|C|} W_C(x+y,x-y)
\nonumber
\\
&= 2^{-n/2} W_C(x+y,x-y)
\nonumber
\\
&= W_C\bigl(2^{-1/2}(x+y), 2^{1/2}(x-y)\bigr):
\label{eq:GII}
\end{align}
the first step uses Theorem~\ref{thm:MacWilliams};
for the second, we deduce $|C| = 2^{n/2}$ from
$2^n = \left| C \right| \cdot |C^\dualcode| = \left| C \right|^2$;
and for the last step, we use the fact that
$W_C$ is a homogeneous polynomial of degree~$n$.  Therefore this
homogeneous polynomial is invariant under the group, call it $G_{\rm II}$,
generated by linear substitutions with matrices
$\slmat{1}{0}{0}{i}$ and $2^{-1/2}\slmat{1}{1}{1}{-1}$.

It turns out that $G_{\rm II}$ is a complex reflection group,
and thus has a polynomial ring of invariants.  Namely, $G_{\rm II}$ is
\#9 in the Shephard-Todd list \cite{ShephardTodd}, and its
invariant degrees are $8$ and~$24$, with $\gleasonphi,\gleasonxi$
as a convenient choice of generators.
This result completes the proof of Gleason's theorem for
Type~II codes.  \end{proof}

In the Appendix we give a direct proof of
$\C[x,y]^{G_{\rm II}} = \C[\gleasonphi,\gleasonxi]$.
The literature contains several other approaches to
the determination of this invariant ring,
including Ebeling's proof in~\cite{Ebeling:lattices}
using the theory of modular forms(!).  See \cite[p.~192]{SPLAG}.
The method we use reaches $G_{\rm II}$ via a suitable tower of
reflection groups starting from $\{1\}$, each normal in the next;
along the way we also obtain Gleason's theorem for Type~I codes,
and encounter a polynomial $\gleasonpsi{12}$,
invariant under an \hbox{index-$2$} subgroup of~$G_{\rm II}$,
that will figure in our subsequent development.

   \section{The Spaces of Discrete Harmonic Polynomials}
      % this is the d_harmonic_poly_section

\label{sec:d_harmonic_poly_section}

In this section, we present some useful results in the theory
of \earlyterm{discrete harmonic polynomials}.  These polynomials
were originally introduced by Delsarte~\cite{Delsarte:Hahn}, who
gave a combinatorial development.  Here, we give a new approach
to these polynomials using the finite-dimensional representation
theory of~$\Slalg_2$.

\subsection{Basic Definitions and Notation}

A function $\monomfunct$ on $\F_2$ may be interpreted as a $2\times
1$ matrix $g=\mono{g_{0}}{g_1}$, where $g_v$ is the value assumed
on input $v\in \F_2$.  It is easily computed that the discrete
Fourier transform $\htr{g}$ of $g$ is the function
$$
\htr{g}=\mono{g_{0}+g_1}{g_0-g_1}=\slmat{1}{1}{1}{-1}\mono{g_{0}}{g_1};
$$
the discrete Fourier transform is therefore encoded by the matrix
$\FToper\defeq \slmat{1}{1}{1}{-1}$.  There is a natural action
of $\Slalg_2$ on these functions $\monomfunct$, defined by
multiplication from the left by matrices in $\Slalg_2$.
Thus, we may interpret the space of functions on $\F_2$ as
a representation of $\Slalg_2$ isomorphic with the
\hbox{$2$-dimensional} defining representation~$\module_1$ of~$\Slalg_2$.

More generally, a monomial function $\monomfunct$ on $\F_2^n$
must have total degree at most~$n$,\footnote{This is a consequence
of the fact that, for any $v\in \F_2^n$, we have $v_j^2=v_j$ for
all $j$ ($1\leq j\leq n$).} and so may be interpreted as a pure
tensor in $\module_1^{\otimes n}$; such a function is denoted
$$
\monomfunct =
\monom{\monomfunct_{10}}{\monomfunct_{11}}{\monomfunct_{n0}}{\monomfunct_{n1}}
$$
and assumes the value $\monomfunct_{1v_1}\cdots\monomfunct_{nv_n}$
on $v\in\F_2^n$.  In this setting, the discrete Fourier transform
corresponds to the action of the operator
\begin{equation}
\FTopern\defeq \FToper^{\otimes n}.
\label{eq:Ttilde}
\end{equation}

For example, the degree-$n$ monomial $\monomfunct_*(v)= v_1\cdots
v_n$, which takes the value of the product of the coordinates
of the input $v\in\F_2^n$, is the function
$$
\monomfunct_*=\monom{0}{1}{0}{1}.
$$
The discrete Fourier transform $\htr{\monomfunct_*}$ of $\monomfunct_*$ is
$$
\htr{\monomfunct}=\FTopern \monomfunct_*=\monom{1}{-1}{1}{-1}.\footnote{
  Note that this aligns with the expression
  $$
  \htr{\monomfunct_*}(u)=\sum_{v\in \F_2^n}(-1)^{\inprodc{u}{v}}
    \monomfunct_*(v)=(-1)^{\sum_{j=1}^nu_j},
  $$%
  obtained from the more common definition \eqref{eq:dft} of the
  discrete Fourier transform given earlier.
  }
$$

\subsubsection{Polynomials in the Variables $(-1)^{v_j}$ ($1\leq j\leq n$)}

Instead of working with polynomials in the variables $v_j$ ($1\leq
j\leq n$), we work with the discrete Fourier transforms  $(-1)^{v_j}$
 ($1\leq j\leq n$) of these variables.\footnote{Delsarte~\cite{Delsarte:Hahn}
uses the $v_j$ basis, rather than the $(-1)^{v_j}$ basis.
% We depart from Delsarte
%%% We'll rhyme some other time
We depart from Delsarte's notation because the use of the $(-1)^{v_j}$ basis
greatly simplifies our development.}  We denote by $\dpolyspace$
the $\C$-vector space of polynomial functions $\dhp$ in the variables
$$
(-1)^{v_1}, \ldots, (-1)^{v_n},
$$
where $v\in \F_2^n$. We denote
by $\dpolyspaced{d}$ the subspace of $\dpolyspace$ consisting
of degree-$d$ homogeneous polynomials in the $(-1)^{v_j}$  ($1\leq
j\leq n$) with each variable $(-1)^{v_j}$ in each term appearing
to degree $0$ or $1$. We adopt the convention that $\dpolyspaced{d}=\{0\}$
for $d<0$.

The preceding discussion shows that any $\dhp\in \dpolyspace$
may be interpreted as an element of~$\module_1^{\otimes n}$, and
that the discrete Fourier transform $\htr{\dhp}$ of $\dhp$ is
equal to $\FTopern\dhp$.  The action of $\Slalg_2$ defined above
gives rise to the following action on $\dpolyspace$: if $M\in\Slalg_2$
and $Q\in \dpolyspace$,  then the action of $M$ on $\dhp$ is given~by
$$
\left(\sumslmateone{M}\right) \dhp.
$$
Here, $\sumslmateone{M}$ denotes the operator equal to
$$
\left(M\otimesdots \idsl\right)+\cdots+\left(\slmatone{M}\right)+\cdots
 + \left(\idsl\otimesdots M\right),
$$
the sum of $n$ tensors, the $j$-th of which acts as $M$ on the $j$-th
factor and as the identity matrix $\slmat{1}{0}{0}{1}$ on
the other factors.

\subsubsection{Conjugation of $\Xoper$, $\Hoper$, and $\Yoper$
by the Discrete Fourier Transform}

Recall that we denote by $(\Xoper,\Hoper,\Yoper)$ the standard basis
for $\Slalg_2$, exhibited in~\eqref{eq:sl2}.
We define the operators $\Xoperp$, $\Hoperp$, and $\Yoperp$ to
be the conjugates of $\Xoper$, $\Hoper$, and $\Yoper$ by the
discrete Fourier transform:
\begin{align}
\Xoperp &\defeq
  \FToper^{-1}\Xoper\FToper=\frac{1}{2}\left(\Hoper-\Xoper+\Yoper\right),
\nonumber  \\
\Hoperp &\defeq
  \FToper^{-1}\Hoper\FToper=\Xoper+\Yoper,
  \label{eq:XHY_conj}
\\
\Yoperp &\defeq
  \FToper^{-1}\Yoper\FToper=\frac{1}{2}\left(\Hoper+\Xoper-\Yoper\right).
\nonumber
\end{align}

Conjugation by the Fourier transform operator $\FToper$ induces
an isomorphism of Lie algebras
\begin{equation}
\Xoper \longleftrightarrow \Xoperp,\quad
\Hoper \longleftrightarrow \Hoperp,\quad
\Yoper \longleftrightarrow \Yoperp,
\label{eq:XHY_invol}
\end{equation}
hence these operators $\Xoperp,\Hoperp,\Yoperp$ satisfy the commutation
relations of~\eqref{eq:slcommrel}:
\begin{equation}\label{eq:ftslcommrel}
\commut{\Xoperp}{\Yoperp}=\Hoperp, \quad
\commut{\Hoperp}{\Xoperp}=2\Xoperp,
\quad \commut{\Hoperp}{\Yoperp}=-2\Yoperp.
\end{equation}

We write $\Xoperpn$, $\Hoperpn$, and $\Yoperpn$ for operators
\begin{align}
\Xoperpn &\defeq \sumslmateone{\Xoperp}
  = \FTopern^{-1} \left(\sumslmateone{\Xoper}\right)\FTopern,
  \nonumber \\
\Hoperpn &\defeq \sumslmateone{\Hoperp}
  = \FTopern^{-1} \left(\sumslmateone{\Hoper}\right)\FTopern,
  \label{eq:XHY_tilde} \\
\Yoperpn &\defeq \sumslmateone{\Yoperp}
  = \FTopern^{-1} \left(\sumslmateone{\Yoper}\right)\FTopern,
\nonumber
\end{align}
which represent the actions of $\Xoperp$, $\Hoperp$
and $\Yoperp$ on elements of $\module_1^{\otimes n}$.  The commutation
relations of~\eqref{eq:ftslcommrel} extend to these operators,
as well:
\begin{equation}\label{eq:ftslcommreln}
\commut{\Xoperpn}{\Yoperpn}=\Hoperpn,
\quad \commut{\Hoperpn}{\Xoperpn}=2\Xoperpn,
\quad \commut{\Hoperpn}{\Yoperpn}=-2\Yoperpn.
\end{equation}
The relations~\eqref{eq:ftslcommreln} induce an isomorphism between
$\Slalg_2$ and the algebra generated by
$\Xoperpn$, $\Hoperpn$, and~$\Yoperpn$.

Now, we have the following result immediately from the definition
of $\Hoperpn$.

\begin{lemma}\label{lem:espacep1}
If $\dhp\in\dpolyspaced{d}$,
then $\Hoperpn\dhp =(n-2d)\dhp$.
\end{lemma}
\begin{proof}
The result follows directly, because the $1$-eigenspace of $\Hoperp$
is the span of $\left\{\mono{1}{1}\right\}$ and the $(-1)$-eigenspace
of $\Hoperp$ is the span of $\left\{\mono{1}{-1}\right\}$.
\end{proof}

For $\dhp\in\dpolyspaced{d}$, we observe that
$\left(\slmatone{\Xoperp}\right)\dhp\in\dpolyspaced{d-1}$,
as we have
$$
\Xoperp\mono{1}{-1}=\mono{1}{1}\quad\text{and}\quad
\Xoperp\mono{1}{1}=\mono{0}{0}.
$$
Thus,
$\Xoperpn\dhp=\left(\sumslmateone{\Xoperp}\right)\dhp\in\dpolyspaced{d-1}$.
We define the \defn{space of degree-$d$ discrete harmonic polynomials} by
\begin{equation}
\dharmspaced{d}\defeq
\ker\left(\Xoperpn:\dpolyspaced{d}\to\dpolyspaced{d-1}\right).
\label{eq:D0_d}
\end{equation}
We then define the  \defn{space of discrete harmonic polynomials},
denoted $\dharmspace$, to be the direct sum
\begin{equation}
\dharmspace\defeq \dispoplus_{d=0}^n
 \dharmspaced{d}=\ker\left(\Xoperpn:\dpolyspace\to \dpolyspace\right).
\label{eq:D0}
\end{equation}

\subsection{Decomposition of Degree-$d$ Discrete Homogeneous Polynomials}

It is immediate from~\eqref{eq:ftslcommreln} that the operator
$\Hoperpn$ maps $\dharmspace$ to itself, since if $\dhp\in \dharmspace$
then
$$
\Xoperpn\Hoperpn \dhp
= \left(\Hoperpn\Xoperpn - \commut{\Hoperpn}{\Xoperpn}\right)\dhp
= \left(\Hoperpn\Xoperpn-2\Xoperpn\right)\dhp=0.
$$

The next lemma substantially refines this observation.
Recall \cite[p.18, Definition 1]{Serre:linear} that an element $e$ of
an $\Slalg_2$ module is said to be \defn{primitive of weight $\lambda$} 
if $e \neq 0$, $\Xoper e = 0$, and $\Hoper e = \lambda e$.

\begin{lemma}\label{lem:espacep}
If $\dhp\in\dharmspaced{d}$, then
$\dhp$ is either zero or primitive of weight $n-2d$
 with respect to the representation of $\Slalg_2$ induced by
 the action of $\Xoperpn$, $\Hoperpn$, and $\Yoperpn$.
\end{lemma}
\begin{proof}
The result is a direct consequence of Lemma~\ref{lem:espacep1}
because all $\dhp\in \dharmspace$ satisfy $\Xoperpn\dhp = 0$.
\end{proof}

\begin{cor}\label{cor:d>n/2}
If $d > n/2$ then $\dharmspaced{d} = \{0\}$.
\end{cor}

\begin{proof}
Since $\dpolyspace$ is finite-dimensional, a primitive vector
must have nonnegative weight.
\end{proof}

For $d\leq \inlinefrac{n}{2}$ and $k=0,1,\ldots, d$, we define
$\dfpolyspaced{d}{k}\defeq (\Yoperpn)^{k}\dharmspaced{d-k}$.\footnote{The
notation $\dfpolyspaced{d}{k}$ is consistent with the notation
$\dharmspaced{d}$ for the space of degree-$d$ discrete harmonic
polynomials.}
Combining Lemma~\ref{lem:espacep} with the representation theory
of~$\Slalg_2$,
we now obtain a decomposition result for $\dpolyspaced{d}$ similar to
that obtained for $\polyspaced{d}$ in Proposition~\ref{prop:harmdecomp}.

\begin{prop}\label{prop:dharmdecomp} For any $d\leq\inlinefrac{n}{2}$,
we have the following results.
\begin{enumerate}
\item The map $\Xoperpn:\dpolyspaced{d}\to \dpolyspaced{d-1}$
is surjective.
\item We have the direct sum decomposition
$\dpolyspaced{d} = \bigoplus_{k=0}^{d}\dfpolyspaced{d}{k}
 = \dharmspaced{d}\oplus \Yoperpn \dpolyspaced{d-1}$.
\item For any nonzero $\dhp\in\dpolyspaced{d}$, the space spanned by
 $\bigl\{(\Yoperpn)^j\dhp\bigr\}_{j=0}^{n-2d}$
 is an irreducible $\Slalg_2$-module isomorphic to
 $\module_{n-2d} := \Sym^{n-2d}(\module_1)$.
\item $\dim(\dharmspaced{d})
  = \dim(\dpolyspaced{d})-\dim(\dpolyspaced{d-1})
  = \binom{n}{d}-\binom{n}{d-1}$.
\end{enumerate}
\end{prop}
\begin{proof}
This follows quickly from Lemma~\ref{lem:espacep} together with
the finite-dimensional representation theory of~$\Slalg_2$;
see for instance \cite[Chapter~IV]{Serre:linear}.
The first and second parts follow from the decomposition of any
finite-dimensional \hbox{$\Slalg_2$-module} as a direct sum of
irreducible modules, together with the explicit action of~$\Slalg_2$
on each of its finite-dimensional irreducible modules
\cite[Chapter~IV, Theorems 2 and~3]{Serre:linear}.
The third part follows from the structure of the irreducible
representation generated by a primitive element of given weight
\cite[Chapter~IV, Corollary~2 of Theorem~1]{Serre:linear}.
The fourth part follows from the first part.
\end{proof}

It also follows that $\Xoperpn : \dpolyspaced{d} \to \dpolyspaced{d-1}$
is \textit{in}jective if $d-1 \geq n/2$, and thus an isomorphism
if $n = 2d-1$; more generally, if $d \geq n/2$ then
$\Xoperpn^{2d-n} : \dpolyspaced{d} \to \dpolyspaced{n-d}$
is an isomorphism.

   \section{The Generalized MacWilliams Identity for Harmonic Weight Enumerators}
      % this is the hwe_section
\label{sec:hwe}

For a length-$n$ binary linear code $C\subset \F_2^n$ and a discrete
harmonic polynomial $\dhp$, the harmonic weight enumerator $W_{C,\dhp}(x,y)$
is defined by\label{def:hwe}
\begin{equation}
W_{C,\dhp}(x,y)\defeq \sum_{c\in C} \dhp(c)x^{n-\wt(c)}y^{\wt(c)}.
\label{eq:hwe_def}
\end{equation}
This function encodes the weights and distribution of the codewords of $C$,
as the weighted theta functions of a lattice $L$ encode the norms and
distribution of the vectors of $L$.

We now derive a generalized MacWilliams identity for harmonic weight
enumerators.
\begin{theorem}
\label{thm:hwsfneq}
For any binary linear code $C\subset \F_2^n$  and $\dhp\in\dharmspaced{d}$,
the harmonic weight enumerator
$W_{C,\dhp}(x,y)=\sum_{c\in C} \dhp(c)x^{n-\wt(c)}y^{\wt(c)}$
satisfies the identity
\begin{equation}
W_{C,\dhp}(x,y) =
  \left(-\frac{xy}{x^2-y^2}\right)^{d}
  \cdot
  \frac{2^{\frac{n}{2}+d}}{|C^\dualcode|}
  \cdot
  W_{C^\dualcode,\dhp}\left(\frac{x+y}{\sqrt 2},\frac{x-y}{\sqrt 2}\right).
\label{eq:gen_McW_Id}
\end{equation}
\end{theorem}
Theorem~\ref{thm:hwsfneq} was first proven by Bachoc~\cite{Bachoc:binary},
via a purely combinatorial argument.  Here, we give a new proof of
this result in analogy with the proof of Theorem~\ref{thm:thetaP_func_eq}.

\subsection{Derivation of the Identity}

For $\dhp\in \dpolyspace$, the function $Q(v)x^{n-\wt(v)}y^{\wt(v)}$
corresponds in the tensor representation to the function
$$
\left(\slmatn{x}{0}{0}{y}\right)Q.
$$
 Therefore, in analogy with the Gaussian operators $\Goper{t}$ defined
in Section~\ref{subsec:wtf_section}, we introduce the operators
\begin{align}
\Woper&\defeq \slmat{x}{0}{0}{y},& \Wopern&\defeq \Woper^{\otimes n},\\
\Wopero&\defeq \slmat{x+y}{0}{0}{x-y}, &\Woperon&\defeq \Wopero^{\otimes
n}.
\label{eq:VW,VWtilde}
\end{align}
The operator $\Wopern$ serves as a sort of ``discrete Gaussian''
for weight enumerators.  Indeed, the weight enumerator
$W_{C}(x,y)$ of a length-$n$ binary linear code is given by
\begin{equation}
W_C(x,y)=\sum_{c\in C}\act{\Wopern\cdot\monomone{1}{1}}{c},
\label{eq:WC_field}
\end{equation}
and the Fourier transform of $\Wopern\cdot\monomone{1}{1}$
is equal to $\Woperon\cdot\monomone{1}{1}$.

\begin{lemma}\label{lem:prop12c}
If $\dhp\in \dpolyspaced{d}$, then we have \begin{equation*}
 \OPopern \dhp= \htr{\dhp},
\end{equation*} where $\htr{\dhp} =\sum_{d'=0}^d\htr{\dhp}_{d'}$ with
$\htr{\dhp}_{d'}\in \dpolyspaced{d}$ for each $d'$ ($0\leq d'\leq d$)
and
\begin{equation}
\label{eq:dtermexpr}
\htr{\dhp}_{d}=\left(-\frac{2xy}{x^2-y^2}\right)^d \dhp.
\end{equation}
\end{lemma}
\begin{proof}
We proceed by strong induction on~$d$.  The base case $d=0$
is immediate, so we suppose that the result holds for
$\dhp\in \dpolyspaced{d_1}$ for each nonnegative $d_1 \leq d$,
and deduce that the result holds also for $\dhp\in \dpolyspaced{d+1}$.

The discrete Fourier transform operator is linear, hence it suffices
to prove the result for the polynomials of the form $(-1)^{v_j}\cdot
\dhp$ with $\dhp\in \dpolyspaced{d}$.  Now, we compute the value of
$\Woperon^{-1}\FTopern$ times
$$
(-1)^{v_j}\cdot \dhp(v)\cdot x^{n-\wt(v)}y^{\wt(v)}=
\Wopern\cdot\left( \slmate{1}{0}{0}{-1}\right)\cdot \dhp
$$
explicitly.  We find that
\begin{align}
\nonumber
\Woperon^{-1}&\FTopern  \left(\Wopern \left(\slmate{1}{0}{0}{-1}\right)
\dhp\right)\\
&=\OPopern \left(\slmate{1}{0}{0}{-1}\right) \OPopern^{-1}\OPopern
\dhp\nonumber\\
&=\left(\slmate{0}{\frac{x-y}{x+y}}{\frac{x+y}{x-y}}{0}\right)\OPopern
\dhp\nonumber\\
&=\left(\slmate{0}{\frac{x-y}{x+y}}{\frac{x+y}{x-y}}{0}\right)\htr{\dhp},
\label{eq:istephwehm}
\end{align}
where the last equality in~\eqref{eq:istephwehm} follows
on applying the inductive hypothesis to $\OPopern \dhp$.

It is clear that the right-hand side of \eqref{eq:istephwehm} has maximal
degree $d+1$, since $\htr{\dhp}$ is of degree~$d$~and
$$
\slmate{0}{\frac{x-y}{x+y}}{\frac{x+y}{x-y}}{0}
$$
is the identity on all but one coordinate.  To finish the proof of
the lemma, we compute the degree-$(d+1)$ term of~\eqref{eq:istephwehm}.
Now, since
$$
\slmat{0}{\frac{x-y}{x+y}}{\frac{x+y}{x-y}}{0}\mono{1}{1}=
\frac{x^2+y^2}{x^2-y^2}\mono{1}{1}-\frac{2 x y}{x^2-y^2}\mono{1}{-1},
$$
the degree-$(d+1)$ term of \eqref{eq:istephwehm} must equal
$-\frac{2 x y}{x^2-y^2}\htr{\dhp}_d$.\footnote{
  Here, $\htr{\dhp}_d$ is the degree-$d$ term of $\htr{\dhp}$,
  as in the lemma statement.
  }
The desired expression~\eqref{eq:dtermexpr}
then follows from the inductive hypothesis.
\end{proof}

\begin{lemma}\label{lem:dmultgauss}
If $\dhp\in \dharmspace$ and $\Hoperpn\dhp=\lambda \cdot \dhp$, then
\begin{enumerate}
\item $\OPopern\Xoperpn\OPopern^{-1}\dhp=0$ and
\item $\OPopern\Hoperpn\OPopern^{-1}\dhp=\lambda \cdot \dhp$.
\end{enumerate}
\end{lemma}
\begin{proof}
Explicit computation gives
\begin{gather}\label{eq:dmutgauss1}
\OPopern\Xoperpn\OPopern^{-1}=-\frac{x^2-y^2}{2xy}\cdot\Xoperpn,\\
\label{eq:dmutgauss2}
\OPopern\Hoperpn\OPopern^{-1}=\Hoperpn+\frac{x^2+y^2}{xy}\cdot\Xoperpn.
\end{gather}
 The first and second results follow directly from \eqref{eq:dmutgauss1}
and \eqref{eq:dmutgauss2}, respectively, since
\begin{align*}
\dhp\in\dharmspace=\ker(\Xoperpn).&\qedhere
\end{align*}
\end{proof}

\begin{cor}\label{cor:dmultgauss} The operators $\Xoperpn$ and $\Hoperpn$
act on $\OPopern\dharmspace$.  The subspace $\OPopern\dharmspaced{d}$
is the intersection of $\ker(\Xoperpn)$ and the $(n-2d)$-eigenspace
of $\Hoperpn+\frac{x^2+y^2}{xy}\Xoperpn$ in $\OPopern\dharmspaced{d}$.
\end{cor}

\subsubsection{Proof of the Generalized MacWilliams Identity}

As a final step en route to Theorem~\ref{thm:hwsfneq}, we
prove an expression analogous to Proposition~\ref{prop:fourier}
for the discrete Fourier transform of the product
of $\Wopern$ and a discrete harmonic polynomial $\dhp\in \dharmspaced{d}$.
\begin{prop}\label{prop:interimfneqhwe}
If $\dhp\in \dharmspaced{d}$, then
\begin{equation}\label{eq:interimfneqhwe}
\OPopern\dhp = \left(-\frac{2xy}{x^2-y^2}\right)^d \dhp.
\end{equation}
\end{prop}
\begin{proof}From Corollary~\ref{cor:dmultgauss}, we see that $\OPopern\dhp$
is in both $\dharmspace$ and (since then $\Hoperpn \dhp = 0$)
the $(n-2d)$-eigenspace of~$\Hoperpn$.
That is, $\OPopern\dhp\in\dharmspaced{d}$.
The result then follows immediately from Lemma~\ref{lem:prop12c}.
\end{proof}

Finally, we obtain the generalized MacWilliams identity by combining
Proposition~\ref{prop:interimfneqhwe} with the discrete Poisson summation
formula (Theorem~\ref{thm:dpsum}).

\begin{proof}[Proof of Theorem~\ref{thm:hwsfneq}]
We obtain the discrete Fourier transform of $\Wopern\dhp$ from
Proposition~\ref{prop:interimfneqhwe}:
\begin{equation}
\label{eq:applypoisson}
\FTopern\bigl(\Wopern\dhp\bigr)
= \left(\frac{-2xy}{x^2-y^2}\right)^d\Woperon\dhp
= \left(\frac{-2xy}{x^2-y^2}\right)^d\cdot
  2^{\inlinefrac{n}{2}}\cdot
  \left(\slmatn{\frac{x+y}{\sqrt{2}}}{0}{0}{\frac{x-y}{\sqrt{2}}}\right)\cdot\dhp.
\end{equation}
The desired formula \eqref{eq:gen_McW_Id} then follows
directly from~\eqref{eq:applypoisson},
upon applying Theorem~\ref{thm:dpsum}.
\end{proof}

\subsubsection*{Remark}
One interesting consequence of Theorem~\ref{thm:hwsfneq} is the fact
that $\inlinefrac{W_{C,\dhp}(x,y)}{(xy)^d}$ is a polynomial,
for any $\dhp\in\dharmspaced{d}$.

\begin{cor}\label{cor:ispoly!}
For $C$ a binary linear code and $\dhp\in \dharmspaced{d}$,
$$
\frac{W_{C,\dhp}(x,y)}{(xy)^d}
$$
is a polynomial in the variables $x,y$.
\end{cor}
\begin{proof}
By Theorem~\ref{thm:hwsfneq},
\begin{equation}
\label{eq:ispoly}
\frac{W_{C,\dhp}(x,y)}{(xy)^d}
=
  \left(-\frac{1}{x^2-y^2}\right)^{d}
  \cdot
  \frac{2^{\frac{n}{2}+d}}{|C^\dualcode|}
  \cdot
  W_{C^\dualcode,\dhp}\left(\frac{x+y}{\sqrt 2},\frac{x-y}{\sqrt 2}\right).
\end{equation}
The left-hand side is a rational function in $x,y$ whose denominator
divides $(xy)^d$, and the right-hand side is a rational function 
whose denominator divides $(x^2-y^2)^d$.
Since $(xy)^d$ and $(x^2-y^2)^d$ are relatively prime,
\eqref{eq:ispoly} is an identity between polynomials in $x$ and~$y$.
\end{proof}
As we see at the end of Section~\ref{sec:assmat}, Corollary~\ref{cor:ispoly!}
also follows directly from the $\Slalg_2$ development of discrete
harmonic polynomials.

\subsection{A Generalization of Gleason's Theorem}

In addition to the generalized MacWilliams identity, Bachoc~\cite{Bachoc:binary}
obtained a harmonic weight enumerator generalization of Gleason's
theorem.  As we will use this result in Section~\ref{sec:assmat},
we state it here.
\begin{theorem}[Bachoc \cite{Bachoc:binary}]
Let $C$ be a Type~II code of length $n$ and let $\dhp\in\dharmspaced{d}$.
 Then, the harmonic weight enumerator $W_{C,\dhp}(x,y)$ is an element
of the principal module $\C[\gleasonphi,\gleasonxi] \gleasonpsi{d}$
for the polynomial algebra $\C[\gleasonphi,\gleasonxi]$,
whose generator is given by
\begin{equation}
\label{eq:psi}
\gleasonpsi{d}\defeq\begin{cases}
1&d\equiv 0\bmod 4,\\
x^3y^3(x^4-y^4)^2(x^8-y^8)(x^8-34x^4y^4+y^8)&d\equiv 1\bmod 4,\\
x^2y^2(x^4-y^4)^2&d\equiv 2\bmod 4,\\
xy(x^8-y^8)(x^8-34x^4y^4+y^8)&d\equiv 3\bmod 4.
\end{cases}
\end{equation}
\label{thm:gleasongen}
\end{theorem}

The degree-$12$ polynomial $\gleasonpsi{2}$ is a square root of
$\gleasonxi$; thus the harmonic enumerators that can arise for even~$d$\/
are elements of the polynomial ring $\C[\gleasonphi,\gleasonpsi{2}]$,
which is the ring of invariants for a complex reflection group
contained with index~$2$ in $G_{\rm II}$ (see the Appendix).
For odd~$d$, the polynomials $\gleasonpsi{d}$ are more complicated
covariants of~$G_{\rm II}$; we have
$\gleasonpsi{1} = \gleasonpsi{2} \gleasonpsi{3}$ and
$\gleasonpsi{3}^2 = \gleasonpsi{2}\0 (\gleasonphi^3 - 108 \gleasonxi\0)$.

   \section{Zonal Harmonic Polynomials}
      % this is the d_zonal_harmonics_section
\label{sec:d_zonal_harmonics_section}

We now introduce the \earlyterm{zonal harmonic polynomials}, a class
$\dzharmspace$ of discrete harmonic polynomials analogous to the zonal
spherical harmonics
mentioned at the end of Section~\ref{sec:wtf_section}.
Specifically, we fix some $\cfixed\in \F_2^n$
and some $d$ with $0\leq d\leq \wtcfixed$,
and determine the space $\dzharmspaced{d}\subset\dharmspaced{d}$
of degree-$d$ discrete harmonic polynomials invariant under coordinate
permutations fixing $\cfixed$.

\subsection{Preliminaries}
Throughout, we fix $\cfixed\in \F_2^n$.
We denote by $\dzpolyspaced{d} \subset \dpolyspaced{d}$
the space of degree-$d$ discrete homogeneous polynomials
invariant under the group of coordinate permutations fixing $\cfixed$,
and set $\dzharmspaced{d}\defeq \dzpolyspaced{d}\cap \dharmspaced{d}$.
We say that a polynomial in $\dzharmspaced{d}$
is a \defn{zonal harmonic polynomial of degree~$d$}, and we define
the space $\dzharmspace$ of \defn{zonal harmonic polynomials} by
\begin{equation}
\dzharmspace\defeq\bigoplus_{d=0}^{\wtcfixed}\dzharmspaced{d}.
\label{eq:zd0}
\end{equation}

\subsubsection{Generators of $\dzpolyspaced{d}$}

We now fix some $d$\/ with $0\leq d\leq \wtcfixed$ and let
$$
\coordone\defeq\{j\setsep \cfixed_j=1\},\quad
\coordzero\defeq\{j\setsep \cfixed_j=0\}.
$$
Now, we denote by $\invgen{d}{k}(v)$ the degree-$d$ discrete polynomial
\begin{align}
\nonumber\invgen{d}{k}(v)&\defeq\sum_{{\{j_1,\ldots,j_k\}\subseteq
  \coordone}\atop{\{j_{k+1},\ldots, j_{d}\}\subseteq \coordzero}}
  (-1)^{(v_{j_1}+\cdots + v_{j_k})+(v_{j_{k+1}}+\cdots +v_{j_{d}})}
\\&=\sum_{{\{j_1,\ldots, j_k\}\subseteq \coordone}\atop{\{j_{k+1},\ldots,
j_{d}\}\subseteq \coordzero}}(-1)^{v_{j_1}}\cdots (-1)^{v_{j_k}}\cdot
(-1)^{v_{j_{k+1}}}\cdots (-1)^{v_{j_{d}}}\in\dpolyspaced{d}.
\label{eq:meetinvgen}
\end{align}
The sum is nonempty for all $d$ ($0\leq d\leq \wtcfixed$)
since $|\coordone|=\wtcfixed$ and $|\coordzero|=n-\wtcfixed$.

By construction, it is clear that $\invgen{d}{k}\in \dzpolyspaced{d}$.
Conversely, we have the following lemma.
\begin{lemma}\label{lem:polygen}
The polynomials $\left\{\invgen{d}{k}\right\}_{k=0}^{\wtcfixed}$
generate $\dzpolyspaced{d}$.
\end{lemma}
\begin{proof}
The result follows immediately from the requirement that any
$\dhp\in \dzpolyspaced{d}$ be invariant under all permutations
simultaneously permuting the $\wtcfixed$ nonzero coordinates of $\cfixed$
and the $n-\wtcfixed$ vanishing coordinates in $\cfixed$, together with
the fact that the multilinear monomials in the variables $(-1)^{v_j}$
are a basis for~$\dpolyspace$.
\end{proof}

Additionally, we have a combinatorial formula for $\invgen{d}{k}(v)$.
\begin{prop}\label{prop:combform}
We have
\begin{multline}
\invgen{d}{k}(v) =
\left(
 \sum_{i=0}^k (-1)^i
  \binom{\wt(\isectcode{v}{\cfixed})}{i}
  \binom{\wtcfixed-\wt(\isectcode{v}{\cfixed})}{k-i}
\right)
  \times\\
\left(
 \sum_{i=0}^{d-k}
  \binom{\wt(v)-\wt(\isectcode{v}{\cfixed})}{i}
  \binom{\left(n-\wtcfixed\right)-\left(\wt(v)-\wt(\isectcode{v}{\cfixed})\right)}{d-k-i}
\right).
\label{eq:Qdkv}
\end{multline}
\end{prop}
\noindent The proof of Proposition~\ref{prop:combform} is immediately
obtained from evaluation of the expression \eqref{eq:meetinvgen}
for~$\invgen{d}{k}$.

\subsubsection{The action of $\Xoperpn$ on $\invgen{d}{k}$}
Now, we determine the action of $\Xoperpn$ on the polynomials
$\left\{\invgen{d}{k}\right\}_{k=0}^{\wtcfixed}$.
\begin{lemma}\label{lem:Xonzone}
We have
\begin{equation}
\Xoperpn \invgen{d}{k}=\bigl((n-\wtcfixed)-(d-k-1)\bigr)\invgen{d-1}{k}
 + \bigl(\wtcfixed-(k-1)\bigr)\invgen{d-1}{k-1}.
\label{eq:XQ_lemma}
\end{equation}
\end{lemma}
\begin{proof}First, we observe that
\begin{equation}\label{eq:Xactonzonterm}
\Xoperpn\cdot \left((-1)^{v_{j_1}+\cdots +v_{j_{d}}}\right)
= \sum_{\ell=1}^d(-1)^{v_{j_0}+v_{j_1}+\cdots
 + v_{j_{\ell-1}}+v_{j_{\ell+1}}+\cdots +v_{j_{d}}+v_{j_{d+1}}},
\end{equation}
where we have used the convention that $v_{j_0}=0=v_{j_{d+1}}$.\footnote{
  To avoid having to adopt this convention, we could have used the
  slightly more standard notation
  $\sum_{\ell=1}^d(-1)^{v_{j_1} + \cdots + \widehat{v_{j_{\ell}}}+\cdots +v_{j_{d}}}$.
  We opt not to use this notation because it conflicts with our
  usage of $\htr{\cdot}$ for the discrete Fourier transform.
  }
It then follows from \eqref{eq:Xactonzonterm} that
$$
\Xoperpn \invgen{d}{k}=\konstant_{k}\cdot\invgen{d-1}{k}+\konstant_{k-1}
 \cdot \invgen{d-1}{k-1}
$$ for constants $\konstant_{k}, \konstant_{k-1}\in\Z$.
To see that
$$
\konstant_{k-1}=\wtcfixed-(k-1),
$$
we observe that each monomial term in $\invgen{d-1}{k}$ can
arise from $\wtcfixed-(k-1)$ different monomial terms in $\invgen{d}{k}$.
Likewise, we obtain
\begin{align*}\konstant_{k}=(n-\wtcfixed)-(d-k-1).&\qedhere\end{align*}
\end{proof}

\subsection{Determination of the Zonal Harmonic Polynomials}

We now combine Lemma~\ref{lem:polygen} and Lemma~\ref{lem:Xonzone}
to characterize $\dzharmspaced{d}$.
\begin{prop}\label{prop:allzonals}
If $\dhp\in \dzharmspaced{d}$, then $\dhp = \konstant_0\cdot
\dzoharmd{d}{\cfixed}$ for some constant $\konstant_0\in\C$,
where
\begin{equation}
\dzoharmd{d}{\cfixed}(v)\defeq
 \sum_{k=0}^d(-1)^{k}
  \left(\prod_{\ell=0}^{k-1}
   \frac{(n-\wtcfixed)-(d-\ell-1)}{\wtcfixed-\ell}\right)
   \invgen{d}{k}(v).
\label{eq:Qprop}
\end{equation}
\end{prop}

\begin{proof} We consider some $\dhp\in \dzharmspaced{d}=\dzpolyspaced{d}\cap
\dharmspaced{d}$.  By Lemma~\ref{lem:polygen}, there exist constants
$\{\konstant_k\}_{k=0}^{\wtcfixed}\subset \C$ such that
$$
\dhp=\sum_{k=0}^{\wtcfixed}\konstant_k\cdot \invgen{d}{k}.
$$
Since $\dhp\in \dharmspaced{d}$, we have
\begin{align*}
0&=\Xoperpn\dhp=\Xoperpn\left(\sum_{k=0}^{\wtcfixed}\konstant_k\cdot
\invgen{d}{k}\right)=\sum_{k=0}^{\wtcfixed}\konstant_k
  \cdot \Xoperpn\invgen{d}{k}\\
&=\sum_{k=0}^{\wtcfixed}\konstant_k \cdot
 \bigl(((n-\wtcfixed)-(d-k-1)\bigr)\invgen{d-1}{k}
  + \bigl(\wtcfixed-(k-1))\invgen{d-1}{k-1}\bigr)
\\
& = \sum_{k=0}^{\wtcfixed}
 \Bigl(\konstant_k\bigl((n-\wtcfixed)-(d-k-1)\bigr)
 + \konstant_{k+1} (\wtcfixed-(k))\Bigr)\invgen{d-1}{k}.
\end{align*}
(The penultimate equality follows from Lemma~\ref{lem:Xonzone}.)
By comparing coefficients, we then obtain
$$
\konstant_{k+1}
= -\frac{(n-\wtcfixed)-(d-k-1)}{\wtcfixed-k} \, \konstant_{k}
$$
for each $k$ ($0\leq k\leq \wtcfixed-1$); the result follows.
\end{proof}
\begin{cor}
For each $d$ ($0\leq d\leq \wtcfixed$), we have $\dim(\dzharmspaced{d})=1$.
\end{cor}

   \section{$t$-Designs and Extremal Type~II Codes}
      \label{sec:Assmus-Mattson}%this is the assmat_section

\label{sec:assmat}

A \defn{$t$-$(n,w,\lambda)$-design} is a (possibly empty\footnote{
  Again we allow $\somedesign=\emptyset$, which is a
  \hbox{$t$-$(n,w,0)$-design} for all~$t$ and~$w$.
  As with spherical designs,
  for most applications only nonempty $\somedesign$ are of interest,
  but allowing empty designs simplifies the statements of
  the results relating codes with combinatorial designs.
  })
collection $\somedesign$ of distinct $w$-element subsets of
$\{1,\ldots, n\}$ with the property that
$|\{S'\in \somedesign\setsep S\subseteq S'\}|=\lambda$
for every $S\subset \{1,\ldots, n\}$ with $|S|=t$.
This generalizes the notion of a \defn{Steiner system},
which is a \hbox{$t$-$(n,w,1)$} design.
For example, the codewords of weight $4$ in the
extended Hamming code form a \hbox{$3$-$(8,4,1)$-design},
and the codewords of weight~$12$ in the extended binary Golay code
form a \hbox{$5$-$(24,12,48)$} design.  We shall see that these are
special cases of behavior common to all extremal Type~II codes.
When $n$, $w$, and $\lambda$ are undetermined or clear from context,
we omit the qualifier ``$(n,w,\lambda)$'' and simply refer to a
$t$-$(n,w,\lambda)$-design as a \defn{$t$-design}.
(See \cite{CvL} for more about $t$-designs, their uses and
their relations with error-correcting codes.)

\subsection{An Equivalent Characterization of $t$-designs}
Each $S'\in \somedesign$ may be represented by its
\defn{indicator vector} $(c_1,\ldots, c_n)$,
in which $c_j=1$ if and only if $j\in S'$.
Thus, a \hbox{$t$-$(n,w,\lambda)$-design $\somedesign$}
corresponds to a subset of the \defn{Hamming sphere of radius $w$},
$$
\hammingsphere_w\defeq\{v\in\F_2^n\setsep \wt(v)=w\}.
$$
We henceforth treat this representation of $\somedesign$ as
completely equivalent to the setwise representation of~$\somedesign$,
using the relevant terminology interchangeably.

We now introduce the following equivalent characterization of $t$-designs.

\begin{prop}\label{prop:equivprop}
A set $\somedesign\subseteq \hammingsphere_w$ is a $t$-design if
and only if
$$
\sum_{v\in \somedesign}\dhp(v)=0
$$
for all $\dhp\in
\inlinecup_{d=1}^t\dharmspaced{d}$.
\end{prop}

Proposition~\ref{prop:equivprop} is equivalent to Theorem~7 of
Delsarte~\cite{Delsarte:Hahn}.
Our development of $\dharmspace$ leads to a new proof of this result,
which we present below.  In Section~\ref{sec:assmatforreal}, we apply
Proposition~\ref{prop:equivprop} to prove a special case of the
Assmus\thmnamesep Mattson theorem~\cite{assmusmattson}.

Throughout this section, we write $\charf_{X}$ for the characteristic
function of the set $X$, and recall that $\Hopern$ denotes
the action of $\Hoper$ on $\module_1^{\otimes n}$,
\begin{equation}
\Hopern\defeq\sum\slmatone{\Hoper}.
\label{eq:Hopern}
\end{equation}
We begin with a lemma regarding projections of functions
$\dhp\in \dpolyspace$ to the Hamming sphere $\hammingsphere_w$.

\begin{lemma}\label{lem:multchar}
For $\dhp\in \dpolyspace$, we have
$\charf_{\hammingsphere_w} \dhp=\proj{n-2w}{\dhp}$,
where $\proj{n-2w}{\dhp}$ is the projection of $\dhp$ to the
$n-2w$ eigenspace of the action of $\Hopern$ on $\module_1^{\otimes n}$.
\end{lemma}
\begin{proof}
This is immediate because the $1$- and $(-1)$-eigenspaces of $\Hoper$
are respectively spanned by $\left\{\mono{1}{0}\right\}$ and
$\left\{\mono{0}{1}\right\}$.
\end{proof}

We now demonstrate Proposition~\ref{prop:equivprop}.

\begin{proof}[Proof of Proposition~\ref{prop:equivprop}]
We denote by $\dhpkspace$ the subset of $\module_1^{\otimes n}$ consisting
of tensor products of $t$ copies of $\mono{0}{1}$ or $\mono{1}{0}$
and $n-t$ copies of $\mono{1}{1}$.   It is clear that $\dhpkspace$
spans $\dpolyspaced{d}$ for any $d$ ($0\leq d\leq t$).
Now, the set $\somedesign$ is a $t$-design if and only if, for all
$\dhpk\in\dhpkspace$,
$$
\vert\hammingsphere_w\vert \inprodch{\charf_{\somedesign}}{\dhpk}
= \vert\somedesign\vert \inprodch{\charf_{\hammingsphere_w}}{\dhpk},
$$
where $\vert\cdot\vert$ is the cardinality function
and $\inprodch{\cdot}{\cdot}$ is the inner product.
It therefore suffices to show that the set
% of restrictions $\{\restr{\dhpk}{\hammingsphere_w}\setsep\dhpk\in \dhpkspace\}$
$\{ \charf_{\hammingsphere_w} \dhpk \setsep\dhpk\in \dhpkspace\}$
is spanned~by
$$
\bigcup_{d=0}^t\{\charf_{\hammingsphere_w}
\setsep\dhp \in \dharmspaced{d}\}.
$$

By the second part of Proposition~\ref{prop:dharmdecomp},
any $\dhpk\in \dhpkspace$ may be written in the form
$$
\dhpk=\sum_{j=0}^t(\Yoperpn)^j \dhp_j,
$$
with $\dhp_j\in \bigoplus_{d=0}^{t-j}\dharmspaced{d}$.
By Lemma~\ref{lem:multchar} and the hypothesis, it then only remains
to demonstrate that $\proj{n-2w}{(\Yoperpn)^j \dhp_j}$ and $\proj{n-2w}{\dhp_j}$
are related by a constant factor: for each $j=0,\ldots, t$, we have
\begin{equation}
	 \proj{n-2w}{(\Yoperpn)^j \dhp_j}=\konstant\cdot \proj{n-2w}{\dhp_j}
	\label{eq:thekonstant}
\end{equation}
for some constant $\konstant$ depending on both $j$ and $t$.

Now, given any $\dhp\in \dharmspaced{d}$, we see by the third part
of Proposition~\ref{prop:dharmdecomp} that the polynomials~$(\Yoperpn)^k\dhp$
($0\leq k\leq n-2d$) span an irreducible representation of $\Slalg_2$
which is isomorphic to~$\module_{n-2d}$.  We may regard this representation
as $(n-2d)$-th homogeneous part of the polynomial algebra
$\C[\upp{0},\upp{1}]$ with generators $\upp{0}, \upp{1}$
and with actions of $\Xoperp,\Hoperp,\Yoperp$ respectively given~by
\begin{equation}
\label{eq:XHY'_dif}
\up{0}\frac{\partial}{\partial \up{1}},\quad
\left(\up{0}\frac{\partial}{\partial \up{0}}
 -\up{1}\frac{\partial}{\partial \up{1}}\right),
\quad \up{1}\frac{\partial}{\partial \up{0}},
\end{equation}
where $\up{0}=\upp{0}+\upp{1}$ and $\up{1}=\upp{0}-\upp{1}$.
With this identification, we may take $\dhp=(\up{0})^{n-2d}$, as
$$
\dhp\in\ker\bigl(\Xoperpn:\dharmspaced{d}\to\dharmspaced{d-1}\bigr).
$$
We now show that $\proj{n-2w}{(\Yoperpn)^k \dhp}$ and $\proj{n-2w}{\dhp}$
are related by a constant factor for each $k$ ($0\leq k\leq n-2d$);
the desired expression~\eqref{eq:thekonstant} follows.
We observe that $\Hoper$ acts as
\begin{equation}
\label{eq:H_dif}
\upp{0}\frac{\partial}{\partial \upp{0}}
-\upp{1}\frac{\partial}{\partial \upp{1}}.
\end{equation}
Therefore, $\proj{n-2w}{\dhp}=\proj{n-2w}{(\upp{0}+\upp{1})^{n-2d}}$
equals $\binom{n-2d}{w-d} \upp{0}^{n-(d+w)}\upp{1}^{w-d}$.
To see this, note that
$\proj{n-2w}{(\upp{0}+\upp{1})^{n-2d}}
= \binom{n-2d}{b_1} \upp{0}^{\konstant_0}\upp{1}^{\konstant_1}
$
with $\konstant_0+\konstant_1=n-2d$ and $\konstant_0-\konstant_1=n-2w$.
 (The latter statement follows from the definition of $\proj{n-2w}{\cdot}$.)
 Likewise,
$$
\proj{n-2w}{(\Yoperpn)^k \dhp}
= \proj{n-2w}{(\Yoperpn)^k (\upp{0}+\upp{1})^{n-2d}}
$$
is the $\upp{0}^{n-(d+w)}\upp{1}^{w-d}$ component of
$(\Yoperpn)^k(\upp{0}+\upp{1})^{n-2d}$.  Since this
component is equal to
$$
\upp{0}^{n-(d+w)}\upp{1}^{w-d}=\proj{n-2w}{\dhp}
$$
up to a constant factor, we are done.
\end{proof}

\subsubsection*{Remarks}

The constant relating $\proj{n-2w}{(\Yoperpn)^k \dhp}$ and $\proj{n-2w}{\dhp}$
in the proof of Proposition~\ref{prop:equivprop} was obtained directly
from the identification of $\bigl\{(\Yoperpn)^k\dhp\bigr\}_{k=0}^{n-2d}$
with $\module_{n-2d}$.  Consequently, this constant is independent
of the choice of $\dhp\in\dharmspaced{d}$.

Proposition~\ref{prop:equivprop} leads to another equivalent characterization
of $t$-designs which makes the analogy between $t$-designs and spherical
$t$-designs explicit.  We have the following corollary, which is
equivalent to Theorem~6 of Delsarte~\cite{Delsarte:Hahn}.
\begin{cor}\label{cor:equivprop2}
A set $\somedesign\subseteq \hammingsphere_w$ is a $t$-design if
and only if
\begin{equation}
\label{eq:alttdes}
\sum_{v\in \somedesign}\dhp(v)
= \frac{\vert\somedesign\vert}{\vert\hammingsphere_w\vert}
  \sum_{v\in \hammingsphere_w}\dhp(v)
\end{equation}
for all $\dhp\in \inlinecup_{d=0}^t\dpolyspaced{d}$.
\end{cor}
\begin{proof}As~\eqref{eq:alttdes} is immediate when $\dhp$ is constant,
the result follows directly from Proposition~\ref{prop:equivprop}
and the second part of Proposition~\ref{prop:dharmdecomp}.
\end{proof}

Finally, we note that the proof of Proposition~\ref{prop:equivprop}
shows that each $\dhp\in\dharmspaced{d}$ is supported on $\inlinecup_{w=d}^{n-d}\hammingsphere_w$.
 This fact leads to a second proof of Corollary~\ref{cor:ispoly!}.
\begin{proof}[Alternate Proof of Corollary~\ref{cor:ispoly!}]
As $\dhp\in \dharmspaced{d}$ is supported on $\inlinecup_{w=d}^{n-d}\hammingsphere_w$,
we know that
$$
W_{C,\dhp}(x,y)
= \sum_{w=0}^{n}\left(\sum_{c\in\codeshell{w}{C}}\dhp(c)\right)x^{n-w}y^{w}
= \sum_{w=d}^{n-d}\left(\sum_{c\in\codeshell{w}{C}}\dhp(c)\right)x^{n-w}y^{w}.
$$
The result then follows immediately.
\end{proof}

\subsection{The Extremal Type~II Code Case of the Assmus\thmnamesep
Mattson Theorem}\label{sec:assmatforreal}

To illustrate the power of Proposition~\ref{prop:equivprop},
we now prove the Assmus\thmnamesep Mattson theorem~\cite{assmusmattson}
in the important special case of an \emph{extremal Type~II code},
that is, a binary linear code~$C$\/ whose minimal (nonzero) weight
$$
\minwt(C) \defeq \min \{ \wt(c) \setsep c \in C, \; c \neq 0 \}
$$
attains the upper bound $4 \lfloor n/24 \rfloor + 4$
derived by Mallows and Sloane~\cite{MallowsSloane} from Gleason's
theorem for Type~II codes.

For $n \equiv 0 \bmod 8$, we define $\cutoff{n}$ by
\begin{equation}
\label{eq:defcut}
\cutoff{n}\defeq \begin{cases}
5&n\equiv 0\bmod 24,\\
3&n\equiv 8 \bmod 24,\\
1&n\equiv 16 \bmod 24.\\
\end{cases}
\end{equation}

\begin{theorem}\label{thm:extremeassmat}
If $C$ is an extremal Type~II code of length $n$,
then $\codeshell{w}{C}$ is a $t$-design for each $t \leq \cutoff{n}$
and any~$w$.
\end{theorem}

By Proposition~\ref{prop:equivprop}, this theorem follows quickly
from the following result, which is slightly more general and is
a coding-theoretic analog of the $r>0$ part of Theorem~\ref{thm:sph_t_des}.

\begin{prop}\label{prop:extremeassmat}
If $C$ is an extremal Type~II code of length~$n$, then for any $w$ and
any choices of $d\in\{1,\ldots, \cutoff{n}\}\cup\{\cutoff{n}+2\}$
and $Q\in \dharmspaced{d}$, we have
$$
\sum_{c\in \codeshell{w}{C}}Q(c)=0.
$$
\end{prop}

Proposition~\ref{prop:extremeassmat} was originally proven by Calderbank
and Delsarte~\cite{CalderbankDelsarte}.  Here, we demonstrate how
Proposition~\ref{prop:extremeassmat} follows quickly from Theorem~\ref{thm:gleasongen}.
This approach is due to Bachoc~\cite{Bachoc:binary}. Our exposition
of this argument slightly expands that of Bachoc~\cite{Bachoc:binary},
which demonstrated only four cases of the result.

\begin{proof}[Proof of Proposition~\ref{prop:extremeassmat}]
We let $d\in \{1,\ldots, \cutoff{n}\}\cup\{\cutoff{n}+2\}$ and $\dhp\in\dharmspaced{d}$.
 Then, we consider the harmonic weight enumerator $W_{C,\dhp}(x,y)$.
 By Theorems~\ref{thm:hwsfneq} and~\ref{thm:gleasongen}, we see that
$\inlinefrac{W_{C,\dhp}(x,y)}{(xy)^d}$ is of the form
$\gleasonxi^{(\minwt(C)-d-\konstant_d)/4} \cdot f$,
where $\konstant_d$ equals the valuation at $y$ of $\gleasonpsi{d}$.
This factor arises because the valuation at $y$ of $W_{C,\dhp}(x,y)$
is at least $\minwt(C)$.

We see that if $W_{C,\dhp}(x,y)$ is nonzero, then it has degree equal to
\begin{equation}
\label{eq:thelast}(n\bmod 24)+4d-24
\end{equation}
if $d\equiv 0\bmod 2$.  Similarly, $f$ has degree
\begin{equation}
\label{eq:thelast2}(n\bmod 24)+4d-36
\end{equation}
if $d\equiv 1\bmod 2$.   Since~\eqref{eq:thelast} and~\eqref{eq:thelast2}
are always negative for $d\in \{1,\ldots, \cutoff{n}\}\cup\{\cutoff{n}+2\}$,
we must have $f\equiv 0$, whence
\begin{align*}
\sum_{w=0}^n\left(\sum_{c\in \codeshell{w}{C}}Q(c)\right)x^{n-w}y^w
= W_{C,\dhp}(x,y)\equiv 0.&\qedhere
\end{align*}
\end{proof}

We note the following special case of Proposition~\ref{prop:equivprop}
which is relevant to our proofs of configuration results in
Section~\ref{sec:ozeki_code_section}.

\begin{cor}\label{cor:extremeassmat}
If $C$ is an extremal Type~II code of length $n$ and $w>0$, then we have
$$
\sum_{c\in \codeshell{w}{C}}\dzoharmd{t}{\cfixed}(c)=0
$$
for any $t\in\{1,\ldots, \cutoff{n}\}\cup\{\cutoff{n}+2\}$.
\end{cor}

\subsubsection*{Remarks}
As Bachoc~\cite{Bachoc:binary} illustrates,
it is possible to prove the full Assmus\thmnamesep Mattson theorem
with a harmonic weight enumerator argument similar to that used in
the proof of  Proposition~\ref{prop:extremeassmat}.  We have focused
on the case of an extremal Type~II code because the full force
of Corollary~\ref{cor:extremeassmat} is required in
Section~\ref{sec:ozeki_code_section}.

    \section{The Koch Condition on Type~II Codes of Length $24$}
    %this is the Niemeier_codes_section

\label{sec:Niemeier_codes}

\subsection{Tetrad Systems}
\label{subsec:tetrads}

For any code~$C$\/ and integer~$w$,
define $\codeshell{w}{C}$ to be the subset of~$C$
consisting of codewords of weight $w$,
and define $\gencodeshell{w}{C}$ to be the linear subcode of~$C$
generated by $\codeshell{w}{C}$.  (This notation is analogous to
that of Ozeki~\cite{Ozeki:48} for lattices.)

For a doubly even code $C\subset \F_2^n$,
the set $C_4$ is called the \emph{tetrad system} of~$C$.
In analogy with the theory of root systems for lattices,
the code $\gencodeshell{4}{C}$ generated by $C_4$ is called the
\emph{tetrad subcode} of $C$, and if $\gencodeshell{4}{C}=C$
then $C$ is called a \textit{tetrad code}.
The irreducible tetrad codes are exactly
\begin{itemize}
  \item the codes $d_{2k}$ ($k \geq 2$), consisting of
   all words $c \in \F_2^{2k}$ of doubly even weight such that
   $c_{2j-1}=c_{2j}$ for each $j=1,2,\ldots,k$\/;
  \item the $[7,3,4]$ dual Hamming code, called $e_7$ in this context; and
  \item the $[8,4,4]$ extended Hamming code, here called $e_8$
\end{itemize}
(see~\cite{Koch87}).  We use the names $d_{2k}$, $e_7$, $e_8$
because the Construction~A lattices $L_{d_{2k}}$, $L_{e_7}$, and $L_{e_8}$
are isomorphic with the root lattices $D_{2k}$, $E_7$, and $E_8$ respectively.

Analogous to the Coxeter number of an irreducible root system,
we define the \emph{tetrad number} $\eta(C)$ of an irreducible
tetrad code~$C$\/ of length~$m$ to be $|C_4|/m$.
A quick computation shows that each of the $m$ coordinates of~$C$\/
takes the value~$1$ on exactly $4\eta(C)$ words in $C_4$, and that
$\eta(d_{2k}) = (k-1)/4$ for each~$k$,
while $\eta(e_7) = 1$ and $\eta(e_8) = 7/4$.

\subsection{Koch's Tetrad System Condition}
Through appeal to the condition of Venkov~\cite{Venkov:24} restricting
the possible root systems of Type~II lattices of rank~$24$,
Koch~\cite{Koch87} obtained a condition on the tetrad systems of Type~II
codes of length $24$.  Specifically, he showed the following result.

\begin{prop}\label{prop:Niemeiertetradsystems}
If $C$ is a Type~II code of length $24$, then $C$ has
one of the following nine tetrad systems:
\begin{equation}
\label{eq:koch_list}
\emptyset,\quad
6d_4,\quad 4d_6,\quad 3d_8,\quad 2d_{12},\quad
d_{24},\quad 2e_7+d_{10},\quad 3e_8,\quad e_8+d_{16}.
\end{equation}
\end{prop}

Koch recovered this condition from the Niemeier~\cite{Niemeier:24}
classification of Type~II lattices of rank~$24$ via Construction~A.
The condition is also a consequence
of the classification of Type~II codes of length~$24$ given by Pless
and Sloane~\cite{PlessSloane1975}.

\subsection{A Purely Coding-Theoretic Proof of Koch's Condition}\label{sec:codthykoch}

Here, we present our proof~\cite{Elkies+Kominers:24} of
Proposition~\ref{prop:Niemeiertetradsystems}
using the theory of harmonic weight enumerators.
This argument is closely analogous to that of Venkov~\cite{Venkov:24}
for the corresponding criterion on root systems of Type~II lattices
of rank~$24$.  We thus begin with a coding-theoretic analog of
\cite[Proposition~1]{Venkov:24}.

\begin{lemma}\label{lem:irredcomps}
If $C$ is a Type~II code of length $24$, then \begin{itemize}
\item either $\tetrads{C}=\emptyset$ or for each $j$  ($1\leq j\leq 24$)
 there exists $c\in \tetrads{C}$ such that $c_j = 1$, and
\item each irreducible component of $\tetradsyst{C}$ has tetrad number
 equal to $\inlinefrac{|\tetrads{C}|}{24}$.
\end{itemize}
\end{lemma}

\begin{proof}[Proof of Lemma \ref{lem:irredcomps}]
For each $j$ ($1\leq j\leq n$), we denote by $\dhpdegone{j}{n}$ the
discrete harmonic polynomial defined~by
\begin{equation}
\dhpdegone{j}{n}(v) \defeq
n\cdot (-1)^{v_j}-\sum_{k=1}^n(-1)^{v_k}\in \dharmspaced{1}.
\label{eq:Q1}
\end{equation}
As in the proof of Proposition~\ref{prop:extremeassmat}, we see that
the harmonic weight enumerator
\begin{equation}
\label{eq:Nmhwe}
W_{C,\dhpdegone{j}{24}}(x,y)
= \sum_{w=0}^{24}
 \left(\sum_{c\in \codeshell{w}{C}}\dhpdegone{j}{24}(c)\right)x^{24-w}y^w
\end{equation}
vanishes for each $j$ ($1\leq j\leq 24$).
We then obtain
\begin{equation}
\label{eq:sum=0w}
\sum_{c\in \tetrads{C}}(8-48c_j)=0
\end{equation}
for each $j$ ($1\leq j\leq 24$),
since the left-hand side of \eqref{eq:sum=0w}
is the $x^{20} y^4$ coefficient of the discrete Fourier transform
of~\eqref{eq:Nmhwe}.  Reorganizing \eqref{eq:sum=0w} shows that
\begin{equation}
\label{eq:tetradcond}
\left| \{c \in \tetrads{C} \setsep c_j = 1 \} \right|
 = \inlinefrac{|\tetrads{C}|}{6}.
\end{equation}
The first part of the lemma then follows.  In the case that $\tetrads{C}\neq
\emptyset$, we also obtain from \eqref{eq:tetradcond} that each irreducible
component of $\tetradsyst{C}$ has tetrad number
$\inlinefrac{\frac{1}{4}|\tetrads{C}|}{6} = \inlinefrac{|\tetrads{C}|}{24}$.
\end{proof}

\subsubsection*{Remark} Since the discrete harmonic polynomial
$\dhpdegone{j}{n}$ has degree $1$ and is invariant under the
coordinate permutations that fix $j$, it is proportional to
the zonal harmonic polynomial $\dzoharmd{1}{\cfixed}$ where
$\cfixed$ is the \hbox{$j$-th} unit vector.

\begin{proof}[Proof of Proposition~\ref{prop:Niemeiertetradsystems}]
As noted in Section~\ref{subsec:tetrads}, there is at most one
tetrad system
with tetrad number $\eta$ for each $\eta\not\in\{1,\inlinefrac{7}{4}\}$,
while for each $\eta \in\{1,\inlinefrac{7}{4}\}$ there are exactly
two tetrad systems with tetrad number $\eta$, with
$\coxcode(d_{10})=\coxcode(e_7) = 1$ and
$\coxcode(d_{16})=\coxcode(e_8) = \inlinefrac{7}{4}$.

Now, Lemma~\ref{lem:irredcomps} implies that if $\tetrads{C}\neq \emptyset$,
then
either $\tetrads{C}$ consists of $\mu$ copies of the tetrad system
$d_{2k}$ for some $\mu$
and $k>1$ such that $\mu\cdot 2k = 24$, or it has one of the following
two forms:
\samepage{
\begin{itemize}
\item $\delta_{10} d_{10} + \varepsilon_7 e_7$,
 with $\varepsilon_7 > 0$ and $10 \delta_{10} + 7 \varepsilon_7 = 24$, or
\item $\delta_{16} d_{16} + \varepsilon_8 e_8$,
 with $\varepsilon_8 > 0$ and  $16 \delta_{16} + 8 \varepsilon_8 = 24$.
\end{itemize}
}
The resulting tetrad systems are precisely
the eight nonempty systems listed in (\ref{eq:koch_list}).
\end{proof}

   \section{Configurations of Extremal Type~II Codes}
      %this is the ozeki_code_section
\label{sec:ozeki_code_section}

Let $C$ be an extremal Type~II code of length $n=8,24,32,
48, 56, 72$, or~$96$.  Set $\minCoz{} = \minwt(C)$, so that
$\minCoz{} = 4, 8, 8, 12, 12, 16$, or $20$ respectively.
We prove that $C$ is generated by $\codeshell{\minCoz{}}{C}$.
Our approach uses the harmonic weight enumerator machinery
developed in Section~\ref{sec:hwe},
following the approach used for lattices in
\cite{Venkov:32}, \cite{Ozeki:32}, \cite{Ozeki:48}, and
\cite{Kominers:56+72+96}.

First, we present a few brief preliminaries.
For any $\cfixed\in \F_2^n$ and any $j$ ($0\leq j\leq n$),
we denote by $\inprodcountc{j}{\cfixed}$ the value
\begin{equation}
\inprodcountc{j}{\cfixed} \defeq
\left|\left\{c\in \gencodeshell{\minCoz{}}{C}
\setsep
\wt(\isectcode{c}{\cfixed})=j\right\}\right|.
\label{eq:N_j}
\end{equation}
For $c\in C^\dualcode$, we must have
$\inprodcountc{j}{c}=0$ for all odd~$j$.

\begin{lemma}\label{lem:toobig}
If $\ccfixed$ is a minimal-weight representative of the class $[\ccfixed]\in
C/\gencodeshell{\minCoz{n}}{C}$ and $c\in \latshell{\minCoz{n}}{C}$,
we have the inequality
$$
\wt(\isectcode{c}{\ccfixed})\leq \frac{\minCoz{n}}{2}.
$$
\end{lemma}
\begin{proof}
This follows quickly, because if
$\wt(\isectcode{c}{\ccfixed})>\inlinefrac{\minCoz{n}}{2}$,
then $[\ccfixed]$ contains a codeword $c + \ccfixed$ of weight
$$
\wt(c + \ccfixed)=\wt(c)+\wt(\ccfixed)-2\wt(\isectcode{c}{\ccfixed})<\wt(\ccfixed).
$$
This contradicts the minimality of $\ccfixed$ in $[\ccfixed]$.
\end{proof}

We now prove our configuration result for Type~II codes of lengths
$n=48$ and~$72$.  The corresponding results for the remaining values of~$n$
are presented in~\cite{Elkies+Kominers:configuration}
and~\cite{Kominers:thesis}.

\begin{theorem}\label{thm:48+72c}
If $C$ is an extremal Type~II code of length $n=48$ or~$72$, then
$$C=\gencodeshell{\minCoz{n}}{C}.$$
\end{theorem}
\begin{proof}
We consider the equivalence classes of $C/\gencodeshell{\minCoz{n}}{C}$
and assume for the sake of contradiction that there is some class
$[\ccfixed]\in C/\gencodeshell{\minCoz{n}}{C}$
with minimal-weight representative
$\ccfixed$ for which $\wt(\ccfixed)=s>\minCoz{n}$.

As $C$ is self-dual, we have
$\inprodcountc{j}{c}=0$ for all odd~$j$.
Additionally, by Lemma~\ref{lem:toobig},
we must have $\inprodcountc{2j'}{\ccfixed}=0$ for
$j' > \inlinefrac{\minCoz{n}}{4}$.
We now develop a system of equations in the
$$
\frac{\minCoz{n}}{4}+1
$$
variables $\inprodcountc{0}{\ccfixed},\inprodcountc{2}{\ccfixed},\ldots,
\inprodcountc{\inlinefrac{\minCoz{n}}{2}}{\ccfixed}$.
One such equation is
\begin{equation}
\label{eq:codesumup}
\inprodcountc{0}{\ccfixed}+\inprodcountc{2}{\ccfixed}+\cdots+
\inprodcountc{\inlinefrac{\minCoz{n}}{2}}{\ccfixed}=\wecoeff{\minCoz{n}}{C};
\end{equation}
Corollary~\ref{cor:extremeassmat} with $\cfixed = \ccfixed$
yields $\cutoff{n}+1$ more.  This yields a system of
$$
\cutoff{n}+2 >  \frac{\minCoz{n}}{4}+1
$$
equations in the variables $\inprodcountc{2j'}{\ccfixed}$
($0\leq j'\leq \inlinefrac{\minCoz{n}}{4}$).

For $n=48, 72$, the (extended) determinants of these inhomogeneous
systems are
\begin{gather}\label{eq:firstdetc1}
 2^{26}3^{5}5^{2}7^{1}11^{2}23^{2}43^{1}47^{1}\left(\frac{ 11 s^3-396
s^2+4906 s-20736}{(s-3) (s-2)^2 (s-1)^3 s^3}\right),\\
2^{42}3^{5}5^{2}7^{2}11^{2}13^{1}17^{3}23^{2}67^{2}71^{1}\left(\frac{
39 s^4-2600 s^3+67410 s^2-800440 s+3650496}{(s-4) (s-3)^2 (s-2)^3 (s-1)^4
s^4}\right),\label{eq:lastdetc1}
\end{gather}respectively\footnote{These determinants were computed
using the formula of Proposition~\ref{prop:combform}.  We omit the
equations obtained from the zonal spherical harmonic polynomials of
the largest degrees when there are more than $\frac{\minCoz{n}}{4}+2$
equations obtained by this method.}; these determinants must vanish,
as they are derived from overdetermined systems.
Since equations~\eqref{eq:firstdetc1} and \eqref{eq:lastdetc1}
have no integer roots $s$, we have reached a contradiction.
\end{proof}

\bibliographystyle{amsalpha}
\bibliography{bibliography/references}

\providecommand{\bysame}{\leavevmode\hbox to3em{\hrulefill}\thinspace}
\providecommand{\MR}{\relax\ifhmode\unskip\space\fi MR }
\providecommand{\MRhref}[2]{%
  \href{http://www.ams.org/mathscinet-getitem?mr=#1}{#2}
}
\providecommand{\href}[2]{#2}
\begin{thebibliography}{Kom09b}

\bibitem[AM69]{assmusmattson}
E.~F. Assmus and H.~F. Mattson, \emph{New 5-designs}, Journal of Combinatorial
  Theory \textbf{6} (1969), 122--151.

\bibitem[Bac99]{Bachoc:binary}
C.~Bachoc, \emph{On harmonic weight enumerators of binary codes}, Designs,
  Codes and Cryptography \textbf{18} (1999), 11--28.

\bibitem[Bac01]{Bachoc:non-binary}
\bysame, \emph{Harmonic weight enumerators of non-binary codes and
  {MacWilliams} identities}, Codes and Association Schemes (A.~Barg and
  S.~Litsyn, eds.), {DIMACS Series in Discrete Mathematics and Theoretical
  Computer Science}, vol.~56, American Mathematical Society, 2001, pp.~1--24.

\bibitem[CD93]{CalderbankDelsarte}
A.~R. Calderbank and P.~Delsarte, \emph{On error-correcting codes and invariant
  linear forms}, {SIAM Journal on Discrete Mathematics} \textbf{6} (1993),
  1--23.

\bibitem[Con69]{Conway:Leech}
J.~H. Conway, \emph{A characterization of {L}eech's lattice}, Inventiones
  Mathematic\ae\ \textbf{7} (1969), 137--142.

\bibitem[CS99]{SPLAG}
J.~H. Conway and N.~J.~A. Sloane, \emph{{Sphere Packings, Lattices and
  Groups}}, 3rd ed., Springer-Verlag, 1999.

\bibitem[CvL91]{CvL}
P.~J. Cameron and J.~H. van Lint, \emph{{Designs, Graphs, Codes and their
  Links}}, Cambridge University Press, 1991, (London Math.\ Society Student
  Texts \textbf{22}).

\bibitem[Del78]{Delsarte:Hahn}
Ph. Delsarte, \emph{Hahn polynomials, discrete harmonics, and $t$-designs},
  SIAM Journal on Applied Mathematics \textbf{34} (1978), 157--166.

\bibitem[Ebe02]{Ebeling:lattices}
W.~Ebeling, \emph{Lattices and codes: A course partially based on lectures by
  {F}.~{H}irzebruch}, 2nd ed., Vieweg, 2002.

\bibitem[EK10]{Elkies+Kominers:24}
N.~D. Elkies and S.~D. Kominers, \emph{On the classification of {T}ype~{II}
  codes of length $24$}, {SIAM Journal on Discrete Mathematics} \textbf{23}
  (2010), no.~4, 2173--2177.

\bibitem[EK11]{Elkies+Kominers:configuration}
\bysame, \emph{Configurations of extremal {T}ype~{II} codes}, {in preparation},
  2011.

\bibitem[Elk00]{NDE:Notices}
N.~D. Elkies, \emph{Lattices, {L}inear {C}odes, and {I}nvariants {I}, {II}},
  Notices of the American Mathematical Society \textbf{47} (2000), 1238--1245
  and 1382--1391.

\bibitem[Elk11]{Elkies:40r}
\bysame, \emph{On the quotient of an extremal {T}ype~{II} lattice of rank $40$,
  $80$, or $120$ by the span of its minimal vectors}, {in preparation}, 2011.

\bibitem[Gle71]{Gleason:gleason}
A.~M. Gleason, \emph{Weight polynomials of self-dual codes and the
  {MacWilliams} identities}, {Actes, Congr\'es International de Math\'ematiques
  (Nice, 1970)}, vol.~3, Gauthiers-Villars, 1971, pp.~211--215.

\bibitem[KAL06]{Putnam05}
L.~F. Klosinski, G.~L. Alexanderson, and L.~C. Larson, \emph{The
  {S}ixty-{S}ixth {W}illiam {L}owell {P}utnam {M}athematical {C}ompetition},
  American Mathematical Monthly \textbf{113} (2006), 733--743.

\bibitem[Kin03]{King:32}
O.~D. King, \emph{A mass formula for unimodular lattices with no roots},
  Mathematics of Computation \textbf{72} (2003), 839--863.

\bibitem[Koc87]{Koch87}
H.~Koch, \emph{Unimodular lattices and self-dual codes}, Proceedings of the
  International Congress of Mathematicians (Berkeley, Calif., 1986), American
  Mathematical Society, 1987, pp.~457--465.

\bibitem[Kom09a]{Kominers:56+72+96}
S.~D. Kominers, \emph{Configurations of extremal even unimodular lattices},
  International Journal of Number Theory \textbf{5} (2009), 457--464.

\bibitem[Kom09b]{Kominers:thesis}
\bysame, \emph{Weighted generating functions and configuration results for
  {Type II} lattices and codes}, Undergraduate Thesis, Harvard University,
  2009, \url{http://www.scottkom.com/articles/kominers_thesis.pdf}.

\bibitem[K{\"{o}}r90]{Korner}
T.~W. K{\"{o}}rner, \emph{{Fourier Analysis}}, Cambridge University Press,
  1990.

\bibitem[Lan75]{Lang:SL2R}
S.~Lang, \emph{{$\SL_2(\R)$}}, {Addison-Wesley}, 1975.

\bibitem[LS71]{LeechSloane}
J.~Leech and N.~J.~A. Sloane, \emph{Sphere packing and error-correcting codes},
  Canadian Journal of Mathematics \textbf{23} (1971), 718--745.

\bibitem[Mac63]{MacWilliams:identity}
F.~J. MacWilliams, \emph{A theorem on the distribution of weights in a
  systematic code}, Bell System Technical Journal \textbf{42} (1963), 79--84.

\bibitem[MOS75]{MallowsOdlyzkoSloane}
C.~L. Mallows, A.~M. Odlyzko, and N.~J.~A. Sloane, \emph{Upper bounds for
  modular forms, lattices and codes}, Journal of Algebra \textbf{36} (1975),
  68--76.

\bibitem[MS73]{MallowsSloane}
C.~L. Mallows and N.~J.~A. Sloane, \emph{An upper bound for self-dual codes},
  Information and Control \textbf{22} (1973), 188--–200.

\bibitem[MS83]{MacWilliamsSloane:theory}
F.~J. MacWilliams and N.~J.~A. Sloane, \emph{{The Theory of Error-Correcting
  Codes}}, 3rd ed., {North-Holland Mathematical Library}, vol.~16,
  North-Holland, 1983.

\bibitem[Nie73]{Niemeier:24}
H.-V. Niemeier, \emph{{Definite quadratische Formen der Dimension~$24$ und
  Diskriminante~$1$}}, Journal of Number Theory \textbf{5} (1973), 142--178
  (German).

\bibitem[Ott99]{Ott:local}
U.~Ott, \emph{Local weight enumerators for binary self-dual codes}, Journal of
  Combinatorial Theory Series A \textbf{86} (1999), 362--381.

\bibitem[Oze86a]{Ozeki:32}
M.~Ozeki, \emph{On even unimodular positive definite quadratic lattices of
  rank~$32$}, Mathematische Zeitschrift \textbf{191} (1986), 283--291.

\bibitem[Oze86b]{Ozeki:48}
\bysame, \emph{On the configurations of even unimodular lattices of rank~$48$},
  Archiv der Mathematik \textbf{46} (1986), 54--61.

\bibitem[PS75]{PlessSloane1975}
V.~Pless and N.~J.~A. Sloane, \emph{On the classification and enumeration of
  self-dual codes}, Journal of Combinatorial Theory, Series A \textbf{18}
  (1975), 313--335.

\bibitem[Rud76]{Rudin}
W.~Rudin, \emph{{Principles of Mathematical Analysis}}, 3rd ed., McGraw-Hill,
  1976.

\bibitem[Ser73]{Serre:course}
J.-P. Serre, \emph{{A Course in Arithmetic}}, Springer-Verlag, 1973.

\bibitem[Ser87]{Serre:linear}
\bysame, \emph{{Complex Semisimple Lie Algebras}}, Springer-Verlag, 1987.

\bibitem[Sie69]{Siegel:extremal}
C.~L. Siegel, \emph{Berechnung von {Z}etafunktionen an ganzzahligen {S}tellen},
  Nachrichten der Akademie der Wissenschaften zu G\"{o}ttingen,
  Mathematisch-Physikalische Klasse,~II \textbf{1969} (1969), 87--102 [pages
  82--97 in {\em Gesammelte Abhandlungen IV}, Berlin: Springer 1979].

\bibitem[Slo77]{Sloane:invariants}
N.~J.~A. Sloane, \emph{Error-{C}orrecting {C}odes and {I}nvariant {T}heory:
  {N}ew {A}pplications of a {N}ineteenth-{C}entury {T}echnique}, American
  Mathematical Monthly \textbf{84} (1977), 82--107.

\bibitem[ST54]{ShephardTodd}
G.~C. Shephard and J.~A. Todd, \emph{Finite unitary reflection groups},
  Canadian Journal of Mathematics \textbf{6} (1954), 274--304.

\bibitem[Ven80]{Venkov:24}
B.~B. Venkov, \emph{On the classification of integral even unimodular
  $24$-dimensional quadratic forms}, Proceedings of the Steklov Institute of
  Mathematics \textbf{148} (1980), 63--74, $\cong$~\cite[Chapter~18]{SPLAG}.

\bibitem[Ven84]{Venkov:32}
\bysame, \emph{Even unimodular {E}uclidean lattices in dimension $32$}, Journal
  of Mathematical Sciences \textbf{26} (1984), 1860--1867.

\bibitem[Ven01]{Venkov:reseaux}
\bysame, \emph{R\'eseaux et designs sph\'eriques}, {R\'eseaux Euclidiens,
  Designs Sph\'eriques et Formes Modulaires}, {Monographies de L'Enseignement
  Math\'{e}matique}, vol.~37, Enseignement Mathematique, Gen\`eve, 2001,
  pp.~10--86 (French).

\end{thebibliography}

   \appendix
   \section{Proof of Gleason's Theorems for Binary Codes}
      % this is the [Gleason] appendix

Let $G_{\rm I}$ be the subgroup of $\GL_2(\C)$ generated by
$\slmat{1}{0}{0}{-1}$ and $2^{-1/2}\slmat{1}{1}{1}{-1}$,
and let $G_{\rm II}$ be the subgroup of $\GL_2(\C)$ generated by
$\slmat{1}{0}{0}{i}$ and $2^{-1/2}\slmat{1}{1}{1}{-1}$.
We have seen, using \eqref{eq:GII} for the second generator,
that if $C$\/ is a binary code of Type~I
(respectively Type~II) then its weight enumerator $W_C$
is contained in the subring of $\C[x,y]$
invariant under linear substitutions with matrices in $G_{\rm I}$
(resp.\ $G_{\rm II}$).  Here we show that the $G_{\rm I}$ invariants
are generated by $x^2+y^2$ and
$\gleasondelta := x^2 y^2 (x^2-y^2)^2$, and the
$G_{\rm II}$ invariants
are generated by $\gleasonphi =  W_{e_8}(x,y) = x^8+14x^4y^4+y^8$
and $\gleasonxi = x^4y^4(x^4-y^4)^4$.  Note that these are consistent
with $G_{\rm I} \subset G_{\rm II}$
because $\gleasonphi = (x^2+y^2)^4 - 4 \gleasondelta$.

We first show that $G_{\rm I}$, and thus also $G_{\rm II}$, contains
the signed permutation subgroup of $\GL_2(\C)$,
which is isomorphic with the eight-element dihedral group
and is generated by $\slmat{1}{0}{0}{-1}$ and $\slmat{0}{1}{1}{0}$.
Indeed\footnote{
  In the coding context we could also show directly that $W_C$ is invariant
  under $\slmat{0}{1}{1}{0}$, that is, that $W_C(x,y) = W_C(y,x)$.
  Any binary linear code~$C$\/ satisfies $W_C(x,y) = W_C(y,x)$
  if and only if $C$ contains the \hbox{all-1s} vector~$\onevec$:
  in the forward direction, the number of weight~$n$ codewords is
  $W_C(0,1)$, while $W_C(1,0) = 1$ always; in the reverse direction,
  translation by $\onevec$ gives for each~$w$ a bijection between the 
  codewords of weight~$w$ and the codewords of weight $n-w$.
  But we noted already that a self-dual code,
  whether of Type~I or Type~II, contains~$\onevec$.
  }
$G_{\rm I} \ni \slmat{1}{0}{0}{-1} = \slmat{1}{0}{0}{i}^2$,
and we calculate that $\slmat{0}{1}{1}{0}$ is the conjugate of
$\slmat{1}{0}{0}{-1}$ by $2^{-1/2}\slmat{1}{1}{1}{-1}$.
Clearly a polynomial in $x,y$ is invariant under the four matrices
$\slmat{\pm1}{0}{0}{\pm1}$ if and only if it is a polynomial in
$x^2$ and~$y^2$.  To be invariant also under $\slmat{0}{1}{1}{0}$
it must be a symmetric polynomial in $x^2$ and~$y^2$.
Thus the invariants under this dihedral group are the polynomials in
$x^2+y^2$ and~$x^2 y^2$.

We can already find the $G_{\rm I}$-invariant subgroup.
Since the involution $2^{-1/2}\slmat{1}{1}{1}{-1}$ fixes $x^2+y^2$ and takes
$x^2 y^2$ to $(x^2-y^2)^2/4$, it follows that the weight enumerator
of a Type~I code is a polynomial in $x^2+y^2$,
$x^2 y^2 + (x^2-y^2)^2/4$, and $x^2 y^2 (x^2-y^2)^2/4$.
Using the identity $x^2 y^2 + (x^2-y^2)^2/4 = (x^2+y^2)^2/4$,
we dispense with the second of those three generators, and
recover Gleason's theorem for self-dual binary codes~$C$\/
(whether of Type~I or Type~II): the weight enumerator of such a code is
a polynomial in $x^2+y^2$ and $\gleasondelta$.
% Note that this is consistent with Theorem~\ref{thm:Gleason} because
% $\gleasonphi = (x^2+y^2)^4 - 4 \gleasondelta$ and
% $\gleasonxi = (x^2+y^2)^4 \gleasondelta^2$.

To find instead to $G_{\rm II}$ invariants,
we next adjoin the matrix $i \slmat{1}{0}{0}{1}$.  We first show
that this matrix is contained in the scalar subgroup of~$G_{\rm II}$.
We claim that the scalars in~$G_{\rm I}$ are the $8$-th roots of unity.
Any scalar matrix $\mu\slmat{1}{0}{0}{1}$ has determinant $\mu^2$,
and our generators of~$G_{\rm II}$ have determinants $i$ and $-1$, so
$\mu\slmat{1}{0}{0}{1} \in G_{\rm II}$ implies $\mu^8 = 1$.
All such $\mu$ appear because $G_{\rm II}$ contains
\begin{equation}
\left(2^{-1/2}\slmat{1}{1}{1}{-1} \slmat{1}{0}{0}{i}\right)^3
= 2^{-3/2} (2+2i) \slmat{1}{0}{0}{1}
= e^{\pi i / 4} \slmat{1}{0}{0}{1}.
\label{eq:3cycle}
\end{equation}
(The invariance of $W_C$ under $e^{\pi i / 4} \slmat{1}{0}{0}{1}$
already shows that $8 \mid n$.)
In particular $G_{\rm II}$ contains~$i\slmat{1}{0}{0}{1}$.  This
transformation fixes $x^2 y^2$ and takes $x^2+y^2$ to $-(x^2+y^2)$.
Hence the polynomials invariant under the signed permutation group
and $i\slmat{1}{0}{0}{1}$ are precisely the polynomials in
$x^2 y^2$ and $(x^2+y^2)^2$.

Let $Q_1 = (x^2+y^2)^2$, $Q_2 = -4 x^2 y^2$, and
$Q_3 = -(Q_1+Q_2) = - (x^2-y^2)^2$.  We next find elements of~$G_{\rm
II}$
that permute the $Q_j$.  One is
$\varsigma := e^{-3 \pi i / 4} \cdot 2^{-1/2}\slmat{1}{1}{1}{-1} \cdot \slmat{1}{0}{0}{i}$,
which is a \hbox{$3$-cycle} contained in~$G_{\rm II}$ by~\eqref{eq:3cycle}.
We calculate that $\varsigma$ permutes the $Q_j$ cyclically.
The other is the diagonal matrix
$e^{\pi i / 4} \slmat{1}{0}{0}{i}$, which takes $Q_2$ to itself and
$Q_1,Q_3$ to each other.  Thus the subring of $\C[x,y]$
invariant under the subgroup of~$G_{\rm II}$ generated by
signed permutations, $i\slmat{1}{0}{0}{1}$, $\varsigma$, and
$e^{\pi i / 4} \slmat{1}{0}{0}{i}$ consists of the polynomials in
$Q_1,Q_2,Q_3$ invariant under arbitrary permutations.
Since the three $Q_j$ are independent but for the relation
$Q_1 + Q_2 + Q_3 = 0$, the
invariant subring is generated by their elementary symmetric functions
of degrees $2$ and~$3$.  We calculate that these are
$$
Q_1 Q_2 + Q_3 Q_1 + Q_2 Q_3 = -\gleasonphi
\quad {\rm and}\quad
% Q_1 Q_2 Q_3 = 4 \gleasonpsi{12},
Q_1 Q_2 Q_3 = 4 \gleasonpsi{2},
$$
where $\gleasonpsi{2} := x^2 y^2 (x^4-y^4)^2$ is
the \hbox{degree-$12$} invariant of~\eqref{eq:psi}.
Thus the invariant subring is $\C[\gleasonphi,\gleasonpsi{2}]$.
Finally the scalar $e^{\pi i / 4}$ fixes $\gleasonphi$ and takes
$\gleasonpsi{2}$ to $-\gleasonpsi{2}$, so the subring of
$\C[\gleasonphi,\gleasonpsi{2}]$ invariant under $e^{\pi i / 4}$
is $\C[\gleasonphi,\gleasonpsi{2}^2]$.
Since $\gleasonpsi{2}^2 = \gleasonxi$, this proves that
any \hbox{$G_{\rm II}$-invariant} polynomial is contained in
is $\C[\gleasonphi,\gleasonxi]$.

While we proved only that $\C[\gleasonphi,\gleasonxi]$ contains
the invariant subring $\C[x,y]^{G_{\rm II}}$, we readily conclude that
$\C[\gleasonphi,\gleasonxi] = \C[x,y]^{G_{\rm II}}$ by verifying that
both $\gleasonphi$ and $\gleasonxi$ are invariant under~$G_{\rm II}$.
This can be checked either by direct computation of the action of
our generators $2^{-1/2}\slmat{1}{1}{1}{-1}$ and $\slmat{1}{0}{0}{i}$,
or by finding Type~II codes $C_n$ of length $n=8$ and $n=24$
such that $W_{C_8} = \gleasonphi$ and
$W_{C_{24}} = \gleasonphi^3 + \alpha \gleasonxi$
for some $\alpha \neq 0$.  We take for $C_8$ the extended Hamming code,
and for $C_{24}$ the extended Golay code or any of the other
indecomposable Type~II codes of length~$24$.

\end{document}